\documentclass[12pt, reqno, twoside, letterpaper]{amsart}
\usepackage{prime-gaps}

\title[Limit points of the sequence of normalized prime gaps]
      {On limit points of the sequence \break 
       of normalized prime gaps}

\author[W. D. Banks]{William D. Banks}

\address{Department of Mathematics, 
         University of Missouri, 
         Columbia MO, USA.}

\email{bankswd@missouri.edu}

\author[T. Freiberg]{Tristan Freiberg}

\address{Department of Mathematics, 
         University of Missouri, 
         Columbia MO, USA.}

\email{freibergt@missouri.edu}

\author[J. Maynard]{James Maynard}

\address{Magdalen College, 
         Oxford, UK.}

\email{maynard@maths.ox.ac.uk}

\date{\today}

\begin{document}

%%%%%%%%%%%%%%%%%%%%%%%%%%%%%%%%%%%%%%%%%%%%%%%%%%%%%%%%%%%%%%%%%%
%%%%%%%%%%%%%%%%%%%%%%%%%%%% ABSTRACT %%%%%%%%%%%%%%%%%%%%%%%%%%%%
%%%%%%%%%%%%%%%%%%%%%%%%%%%%%%%%%%%%%%%%%%%%%%%%%%%%%%%%%%%%%%%%%%

\begin{abstract}
Let  $p_n$  denote the $n$th smallest prime number, and let  $\LP$ 
denote the set of limit points of the sequence 
$
\Br{(p_{n+1} - p_n)/\log p_n}_{n = 1}^{\infty}
$
of normalized differences between consecutive primes.
We show  that  for $k = 9$ and for any sequence of $k$ nonnegative 
real numbers  $\beta_1  \le \beta_2 \le \cdots  \le  \beta_k$,  at 
least one of the numbers $\beta_j - \beta_i$ ($1 \le i < j \le k$) 
belongs to $\LP$.
It follows that at least $12.5\%$ of all nonnegative  real numbers 
belong to $\LP$.
\end{abstract}

\maketitle

%%%%%%%%%%%%%%%%%%%%%%%%%%%%%%%%%%%%%%%%%%%%%%%%%%%%%%%%%%%%%%%%%%
%%%%%%%%%%%%%%%%%%%%%%%%%%%  SECTION 1 %%%%%%%%%%%%%%%%%%%%%%%%%%%
%%%%%%%%%%%%%%%%%%%%%%%%%%%%%%%%%%%%%%%%%%%%%%%%%%%%%%%%%%%%%%%%%%

\section{Introduction}
 \label{sec:1}
Let  $p_1 = 2 < p_2 = 3 < p_3 = 5 < \cdots$ be the sequence of all 
primes.
The prime  number  theorem  asserts that  $p_n \sim n\log p_n$  as 
$n \to \infty$, hence the $n$th prime gap 
\[
 d_n = p_{n+1} - p_n
\]
is of length approximately $\log p_n$ on average.
It is  natural to ask how  often the  normalized  $n$th  prime gap
$d_n/\log p_n$ lies between two given numbers $\alpha$ and 
$\beta$.
For fixed $\beta > \alpha \ge 0$, heuristics based on Cr{\'a}mer's 
probabilistic model for primes lead to the conjecture that
\begin{align}
 \label{eq:1.1}
 N^{-1}
  \big|
       \big\{
           n \le N : d_n/\log p_n \in (\alpha,\beta] 
       \big\} 
  \big|
   \sim \int_{\alpha}^{\beta} \e^{-t} \dd{t}
    \qquad (N \to \infty).
\end{align}
Thus, one expects that the  normalized prime gaps are  distributed
according to a Poisson process, and the probability that  $d_n$ is 
close to $t\log p_n$ is about $\e^{-t}$.
We refer the reader to  the  expository  article  \cite{SOUND}  of 
Soundararajan  for   further   discussion  of   these  fascinating 
statistics.

Gallagher  \cite{GAL} has shown that  \eqref{eq:1.1}  follows from 
the truth of a suitable  uniform  version of the Hardy--Littlewood  
prime  $k$-tuple  conjecture;  however, such results must lie very 
deep.
An approximation to \eqref{eq:1.1} is the conjecture%
\footnote{%
Erd{\H o}s \cite[p.4]{ERD} wrote: ``It seems certain that 
$d_n/\log n$  is everywhere dense in the interval  $(0,\infty)$.''
}
of Erd{\H o}s \cite{ERD} that if $\LP$  is the set of limit points 
of the sequence 
$
\Br{d_n/\log p_n}_{n = 1}^{\infty},
$
then $\LP = [0,\infty]$.
It had already been established by Westzynthius \cite{WES} in 1931 
that 
\[
 \limsup_{n \to \infty} \frac{d_n}{\log p_n} = \infty.
\]
In 2005, the groundbreaking work of 
Goldston--Pintz--Y{\i}ld{\i}r{\i}m  \cite{GPY} established for the 
first time that
\[
 \liminf_{n \to \infty} \frac{d_n}{\log p_n} = 0.
\]
Hence, $0 \in \LP$  and $\infty \in \LP$, but no other limit point 
of $\LP$ is known at present.

The prime number theorem implies the existence of a limit point in 
$\LP$ that is less than or equal to $1$.
Erd{\H o}s \cite{ERD} and Ricci  \cite{RIC} were able to show that 
$\LP$ has positive Lebesgue  measure, but were unable to show that  
$\LP$ contains a limit point greater than $1$.
Hildebrand  and  Maier \cite{HM} were the first to show that $\LP$ 
contains a limit point greater than $1$.
Indeed, they  showed that  there is a positive constant  $c$  such 
that $\lambda([0,T] \cap \LP) \ge c\,T$ holds for all sufficiently 
large  $T$, where $\lambda$ denotes the Lebesgue measure on $\RR$, 
and hence that $\LP$  contains arbitrarily large limit points.
In  fact,  Hildebrand and Maier proved an $m$-dimensional analogue 
of this  result  for  the  limit  points of   ``chains''   of  $m$ 
consecutive gaps between primes (see Theorem \ref{thm:1.3} below). 

Using the recent breakthrough work of Zhang \cite{ZHA} on  bounded 
gaps between  primes, Pintz   \cite{PIN} has shown that there is a 
small (ineffective) positive constant $c$ such that 
$\LP \supseteq [0,c]$.
Most  recently, Goldston  and  Ledoan  \cite{GL}  have  shown that 
Erd{\H o}s'   method  yields   infinitely  many  limit points   in 
intervals of the form  $[(1/\mathcal{C})(1 - (1/M) - \epsilon),M]$
for any $M > 1$,  where  $\mathcal{C}$  is an  overestimate in the 
sieve  upper  bound for the number of generalized twin primes (one 
can take $\mathcal{C} = 4$).
Further, Goldston  and Ledoan have shown that there are infinitely 
many limit points in intervals such as $[1/2000,3/4]$.

In this paper, we prove the following.

\begin{theorem}
 \label{thm:1.1}
Let $d_n = p_{n+1} - p_n$, where  $p_n$ denotes the $n$th smallest 
prime, and let $\LP$ be the set of limit points of 
$
\Br{d_n/\log p_n}_{n = 1}^{\infty}.
$
For any sequence of $k = 9$ nonnegative real numbers 
$\beta_1 \le \beta_2 \le \cdots \le \beta_k$, we have
\begin{align}
 \label{eq:1.2}
  \big\{ \beta_j - \beta_i : 1 \le i < j \le k \big\} \cap \LP
   \ne \emptyset. 
\end{align}
\end{theorem}

We have the following corollary, which shows that at least 
$12.5\%$ of nonnegative real numbers belong to $\LP$.

\begin{corollary}
 \label{cor:1.2}
Let $\LP$ be as in Theorem \ref{thm:1.1}, and let $\lambda$ be the 
Lebesgue measure on $\RR$.
The following bound holds 
\textup{(}with an ineffective $\oh[1]$\textup{):}
\begin{align}
 \label{eq:1.3}
  \lambda([0,T] \cap \LP ) 
   \ge 
    (1 - \oh[1]) T/8 \qquad (T \to \infty).
\end{align}
The following effective bound also holds\textup{:}
\begin{align}
 \label{eq:1.4}
  \lambda([0,T] \cap \LP ) > T/22 \qquad (T > 0).
\end{align}
\end{corollary}

\begin{proof}
We first observe  that the set  $\LP$, being a countable number of 
unions and intersections of open balls, is Lebesgue measurable.

Now let  $\kappa \ge 2$  be the smallest integer such that for any 
sequence of $\kappa$ real numbers  
$\alpha_{\kappa} \ge \cdots \ge \alpha_1 \ge 0$, we have
\[
 \{\alpha_j - \alpha_i : 1 \le i < j \le \kappa\} 
  \cap \LP 
   \ne \emptyset.
\]
By  Theorem  \ref{thm:1.1} such a  $\kappa$  exists and is at most 
$9$. 
If $\kappa = 2$ then $\LP = [0,\infty]$.
If $\kappa  \ge  3$ then by minimality there is a sequence of real 
numbers 
$\hat{\alpha}_{\kappa - 1} \ge \cdots \ge \hat{\alpha}_1 \ge 0$ 
such that 
\[
 \{\hat{\alpha}_j - \hat{\alpha}_i : 1 \le i < j \le \kappa - 1\} 
  \cap \LP 
 = \emptyset.
\]
Then for any number $\alpha \ge \hat{\alpha}_{\kappa - 1}$, 
$
 \{\alpha - \hat{\alpha}_j : 1 \le j \le \kappa - 1\} 
  \cap \LP 
   \ne \emptyset,
$
that is, for any $T_2 \ge T_1 \ge \hat{\alpha}_{\kappa - 1}$, 
\[
 [T_1, T_2]
  = 
   {\textstyle \bigcup_{j = 1}^{\kappa - 1} } \hspace{3pt}
    \{ 
     \beta + \hat{\alpha}_j : 
      \beta 
       \in 
        [T_1 - \hat{\alpha}_j, T_2 - \hat{\alpha}_j] \cap \LP  
    \}. 
\]
Thus,  by  subadditivity  and  translation  invariance of Lebesgue 
measure, 
\begin{align*}
T_2 - T_1
 & \le \textstyle
    \sum_{j = 1}^{\kappa - 1} 
     \lambda([T_1 - \hat{\alpha}_j, 
              T_2 - \hat{\alpha}_j] \cap \LP)
      \le (\kappa - 1)\lambda([0,T_2] 
            \cap \LP).
\end{align*}
This gives \eqref{eq:1.3}.

With $\kappa$ as above we have
\[
 \{\alpha, 2\alpha,\ldots, (\kappa-1)\alpha\} 
  \cap \LP 
   \ne \emptyset
\]
for every real number $\alpha \ge 0$ (take 
$\hat{\alpha}_j = j\alpha$ for $1 \le j \le \kappa$).
For any  $T \ge 0$,  by subadditivity and positive  homogeneity of 
Lebesgue measure, we have 
\begin{align*} 
     \lambda([0,T]) 
    & \textstyle 
       \le
        \sum_{j=1}^{\kappa-1}\lambda([0,T] \cap j^{-1} \LP)  
      \textstyle =
       \sum_{j=1}^{\kappa-1}j^{-1}\lambda([0,jT] \cap \LP) \\
    & \textstyle \le
         \lambda([0,(\kappa-1)T] \cap \LP)
          \sum_{j=1}^{\kappa-1}j^{-1}.
\end{align*}
Replacing $T$ by $(\kappa-1)^{-1}T$ and recalling that 
$\kappa \le 9$, this gives \eqref{eq:1.4}.
\end{proof}

We actually prove the following more general result on ``chains'' 
of gaps between primes, for which Theorem \ref{thm:1.1} is a 
stronger version of the special case $m = 1$.

\begin{theorem}
 \label{thm:1.3}
Let $d_n = p_{n+1} - p_n$, where  $p_n$ denotes the $n$th smallest 
prime. 
Fix an  integer  $m \ge 2$, and  let $\LP_m$ be the  set of limit 
points in $[0,\infty]^{m}$ of 
\[
 \textstyle 
 \big\{
  \big(
   \frac{d_n}{\log p_n},
    \ldots,
     \frac{d_{n + m - 1}}{\log p_{n + m - 1}}
  \big)
  \big\}_{n = 1}^{\infty}.
\]
Given $\bb = (\beta_1,\ldots,\beta_k) \in \RR^k$, let $S_m(\bb)$ 
be the set
\[
    \Br{
         \br{\beta_{J(2)} - \beta_{J(1)},\ldots,
           \beta_{J(m+1)} - \beta_{J(m)}} :
          1 \le J(1) < \cdots < J(m+1) \le k
       }.
\]
For any  sequence of $k = 8m^2 + 8m$ nonnegative real numbers 
\[
 \beta_1 \le \beta_2 \le \ldots \le \beta_{8m^2 + 8m},
\]
we have
\begin{align}
 \label{eq:1.5} 
 S_m(\bb) \cap \LP_m \ne \emptyset.
\end{align}
\end{theorem}

\begin{acknowledgements*}
For  helpful comments,  corrections or discussions we are grateful 
to 
Daniel Goldston, 
Andrew Granville,
Paul Pollack and
Terence Tao.
\end{acknowledgements*}

%%%%%%%%%%%%%%%%%%%%%%%%%%%%%%%%%%%%%%%%%%%%%%%%%%%%%%%%%%%%%%%%%%
%%%%%%%%%%%%%%%%%%%%%%%%%%%  SECTION 2 %%%%%%%%%%%%%%%%%%%%%%%%%%%
%%%%%%%%%%%%%%%%%%%%%%%%%%%%%%%%%%%%%%%%%%%%%%%%%%%%%%%%%%%%%%%%%%

\section{Notation}
 \label{sec:2}

The set of all primes is  denoted  by $\PP$,  the  $n$th  smallest 
prime by $p_n$,  the  $n$th difference  $p_{n + 1} - p_n$  in  the 
sequence of primes by $d_n$, and $p$ always stands for a prime. 
The indicator function for $\PP$ is denoted $\ind{\PP}$.
The   Euler, von  Mangoldt  and  $k$-fold  divisor  functions  are 
denoted  by  $\phi$,  $\Lambda$  and  $\tau_k$, the prime counting 
functions by $\pi(x) = \sum_{n \le x} \ind{\PP}(n)$, 
$\psi(x) = \sum_{n \le x} \Lambda(n)$,  
 
\[
 \pi(N;q,a) = \sums[n \le N][n \equiv a \bmod q] \ind{\PP}(n), 
  \quad 
 \psi(N;q,a) = \sums[n \le N][n \equiv a \bmod q] \Lambda(n).
\]

A Dirichlet character to the modulus $q$ is denoted $\chi \bmod q$ 
or simply $\chi$, and the $L$-function associated with it is 
denoted $L(s,\chi)$.

The $n$th iterated logarithm is denoted by  $\log_n x$ and defined 
recursively as follows:
$\log_1 x = \max\{1,\log x\}$ and  
$\log_{n+1} x = \log_1(\log_n x)$ for $n \ge 1$.

The  greatest prime factor of an integer $q$ is denoted  $\gp{q}$.

For any  $k$-tuple of integers  $\A  =  \{h_1,\ldots,h_k\}$ and an 
integer $n$, the translated $k$-tuple 
$\{n + h_1,\ldots,n + h_k\}$ is denoted by $\A(n)$.

Throughout,  $c$  denotes a positive  constant that  may differ at 
each occurrence.
Expressions of the  form  $A = \Oh[B]$,  $A \ll B$  and  $B \gg A$  
signify  that  $|A| \le c|B|$.
If $c$ depends on certain parameters this may be indicated by 
subscripts (as in $A \ll_{\epsilon} B$, etc.).
The relation $A \ll B \ll A$ is denoted $A \asymp B$.
Finally,  $\oh[1]$  denotes  a  quantity that  tends to  $0$  as a 
certain parameter (clear in context) tends to infinity.

%%%%%%%%%%%%%%%%%%%%%%%%%%%%%%%%%%%%%%%%%%%%%%%%%%%%%%%%%%%%%%%%%%
%%%%%%%%%%%%%%%%%%%%%%%%%%%  SECTION 3 %%%%%%%%%%%%%%%%%%%%%%%%%%%
%%%%%%%%%%%%%%%%%%%%%%%%%%%%%%%%%%%%%%%%%%%%%%%%%%%%%%%%%%%%%%%%%%

\section{Outline of the proof}
 \label{sec:3}

For the sake of  exposition we ignore  (only in  this section) the 
possibility of Siegel zeros%
\footnote{%
We are abusing terminology here.
By a Siegel zero we mean a real, simple zero of a Dirichlet 
$L$-function (corresponding to a primitive character), in a region 
that we can show is otherwise zero free.
In  some cases  this is a wider region than the  classical one --- 
see Lemma \ref{lem:4.1} below.
}
--  accounting  for  this  possibility  introduces  certain  minor 
technical complications in parts of the proof.

The idea  underlying  our  proof  of  Theorem  \ref{thm:1.3} is to 
combine a  construction  of  Erd{\H o}s  \cite{ERD1935} and Rankin 
\cite{RAN}  with  the  recent  theorem of  Maynard \cite{MAY} and%
\footnote{%
Tao  (unpublished)  independently  discovered  the  same method as 
Maynard around the same time.
}
Tao.

The  Erd{\H o}s--Rankin  construction   produces  long   intervals 
$(n,n + z]$ containing only composite integers.
This is accomplished by choosing a set of integers 
$\{a_p : p \le y\}$, one for each prime $p \le y < z$, so that for 
every integer $g \in (0,z]$, the congruence $g \equiv a_p \bmod p$ 
holds for at least one $p \le y$.
By the  Chinese remainder  theorem one  can  find  an integer $b$, 
uniquely determined modulo $P(y)  =  \prod_{p \le y} p$, such that 
$b \equiv -a_p \bmod p$ for every $p \le y$.
Now suppose  $n \equiv b \bmod P(y)$ and $n > y$.
For any  $g \in (0,z]$  we  have  $g \equiv a_p \bmod p$  for some 
$p \mid P(y)$,  and so  $g + n \equiv a_p - a_p \equiv 0 \bmod p$; 
hence, $g + n$ is composite for each $g\in(0,z]$.
In this  situation  we  say  that  the  progression $b \bmod P(y)$  
{\em  sieves  out}  intervals  of  the  form  $(n, n + z]$,  where 
$n \equiv b \bmod z$ and $n > y$.
Noting that  $\log P(y)  \sim  y$ by the prime number theorem, the 
goal is to maximize the ratio $z/y$.

The  Maynard--Tao  theorem  establishes,  for  the first time, the 
existence of  $(m+1)$-tuples of  primes in  $k$-tuples of integers 
the form 
\[
 \A(n) = \{n + h_1,\ldots,n + h_k\},
\] 
whenever 
$\A = \{h_1,\ldots,h_k\}$ is an  {\em  admissible}  $k$-tuple (see 
\eqref{eq:4.1})  and  $k$  is  large  enough in terms of  $m$, say 
$k \ge k_m$.
The  prime  $k$-tuple   conjecture   asserts  that  one  can  take 
$k_m = m + 1$,  but  since  in   the  Maynard--Tao  theorem  $k_m$  
is exponential in  $m$  (and this seems to be the limit of the 
method of proof at present), no given admissible $(m + 1)$-tuple 
$\A = \{h_1,\ldots,h_{m+1}\}$ is known to give  
$|\A(n) \cap \PP| = m + 1$ for infinitely many $n$.

It turns out that in the Maynard--Tao theorem one can restrict $n$
to lie in an arithmetic  progression --- in fact this is a feature 
of its proof.
Given   a   sufficiently   large   number   $N$   and   a  modulus 
$W = \prod_{p \le w} p$, where $w$  grows slowly with $N$, one can  
take  $n  \in  (N,2N]$  with  $n \equiv b \bmod W$,  provided that  
$b$ is an integer for which $(b + h_i,W) = 1$  for each $i$.
Choosing  the  progression  $b \bmod W$  carefully, one can use it 
to sieve  out all  integers in  intervals of the form  $(n,n + z]$ 
with  $n  \equiv  b  \bmod W$  {\em  except}  for the  integers in 
$\A(n)$.
Used in this way, the Maynard--Tao theorem produces 
{\em consecutive} $m$-tuples of primes in intervals of 
{\em bounded length}.

In  the  present  paper, we  modify  the  above  ideas  to  obtain 
consecutive  primes  in  $\A(n)  =  \{n + h_1, \ldots, n + h_k\}$, 
$n \in (N,2N]$, with differences $h_j - h_i \asymp \log N$.
To do this, we give a uniform version of the  Maynard--Tao theorem 
in  which the elements of the  $k$-tuple $\A = \{h_1,\ldots,h_k\}$  
are allowed to grow with $N$, and in  which $w$ can be as large as  
$\epsilon\log N$ for a sufficiently small $\epsilon$.
This means that the modulus  $W$  is as large as a small  power of 
$N$,  and  for  reasons  concerning  level  of  distribution  (see 
\eqref{eq:4.2}  {\em et seq.}), this extension of the Maynard--Tao 
theorem requires a modification of the Bombieri--Vinogradov 
theorem%
\footnote{%
Putative  Siegel  zeros  have an impact here, and any  exceptional 
moduli must be taken into account.
}
that exploits the fact that the arithmetic progressions with which 
we are concerned have moduli  that are all multiples of the smooth 
integer $W$.

To obtain stronger quantitative results, we use a further 
modification of the Maynard--Tao theorem, which might be of 
independent interest. 
We show that given a partition 
$\A = \A_1 \cup \cdots \cup \A_{8m + 1}$ 
of $\A$ into $8m + 1$ equal sized subsets, there are infinitely 
many $n$ such that $|\A_j(n) \cap \PP| \ge 1$ for at least $m + 1$ 
different values of $j \in \{1,\ldots,8m + 1\}$, provided that the 
size of $\A$ is sufficiently large.

We use a slight modification of the Erd{\H o}s--Rankin 
construction to find an  arithmetic  progression  $b \bmod W$ that 
sieves out the integers in an interval  $(0,z]$,  {\em except} for 
precisely $k$ integers $\A  = \{h_1,\ldots,h_k\}  \subseteq (0,z]$  
that constitute our admissible $k$-tuple.
We want to choose $\A$ so that 
$h_j - h_i \sim (\beta_j - \beta_i)\log N$ for 
$1 \le i < j \le k$,  where $\beta_k \ge \cdots \ge \beta_1 \ge 0$ 
are given.

As in the Erd{\H o}s--Rankin  construction, we select the integers 
$\{a_p : p \le y\}$,  $y \le z$,  in  stages  according  to  their 
size.%
\footnote{%
The effect of any  Siegel zero  here would mean that here we  must 
actually select integers $\{a_p : p \le y, \, p \not\in \Z\}$  for 
a certain sparse set of primes $\Z$.
}
We take  $0 < y_1 < y_2 < y < z$, say, where  $y_1$ and  $y_2$ are 
parameters to be chosen optimally later.
First, we put $a_p = 0$ for primes $p \in (y_1,y_2]$.
Next, we use a ``greedy sieve'' to choose the  $a_p$ optimally for
the small primes  $2 < p \le y_1$, that is, we successively choose 
$a_p$ so that  $g  \equiv  a_p \bmod  p$  for the maximum possible 
number of $g \in (0,z]$ that have remain ``unsifted'' thus far.
Since we do not know the congruence classes  $a_p \bmod p$ for the
smallest primes,  our  approach  does  not work in general for all
$k$-tuples $\A = \{h_1,\ldots,h_k\}$; we find it convenient to 
select our  $k$-tuple  only after sieving by primes $p \le y_2$.
We choose the numbers $h_i$ from among the primes in $(y,z]$.
(This is  why we do  not use  $p = 2$  ``greedily''  --  if we had 
$a_2 = 1$ then only even integers would remain unsifted.)
It  is  clear  that  each  $h_i \not\equiv  a_p \bmod p$  for  all
$p \in (y_1,y_2]$ since for those primes we have $a_p = 0$.
We can also guarantee that  $h_i \not\equiv a_p \bmod p$  for  the 
small primes $p \le y_1$  if we select primes $h_i$  in a suitable 
arithmetic progression $b \bmod P_1$, where 
$P_1 = \prod_{2 < p \le y_1} p$.
We choose $y_1 = (\log y)^{1/4}$, so such  primes exist by (Page's 
version of) the prime number theorem for arithmetic progressions.%
\footnote{%
Again, this is assuming Siegel zeros do not exist.
If  they do, we  need  only  discard  at  most  one prime from the 
product  defining  $P_1$  to  ensure it  isn't  a  multiple  of an  
``exceptional'' modulus.
}%
\footnote{%
The reader will note that $y_1 = (\log y)^{1/4}$  is  smaller than 
the optimal choice for  $y_1$  in  the original Erd{\H o}s--Rankin 
construction. 
By a more  careful  argument one may be able to take $y_1$ larger,
but this is not necessary for our  application, and  so we satisfy 
ourselves with a smaller choice of $y_1$.
}

%%%%%%%%%%%%%%%%%%%%%%%%%%%%%%%%%%%%%%%%%%%%%%%%%%%%%%%%%%%%%%%%%%
%%%%%%%%%%%%%%%%%%%%%%%%%%%  SECTION 4 %%%%%%%%%%%%%%%%%%%%%%%%%%%
%%%%%%%%%%%%%%%%%%%%%%%%%%%%%%%%%%%%%%%%%%%%%%%%%%%%%%%%%%%%%%%%%%

\section{A uniform Maynard--Tao theorem}
 \label{sec:4}

\subsection{Preliminaries}
 \label{subsec:4.1}
A precise  statement of the  version of  Maynard--Tao that we will 
use requires some notation, terminology and setting up.

We say that a given $k$-tuple of integers 
$\A = \{h_1,\ldots,h_k\}$ is {\em admissible} if
\begin{align}
 \label{eq:4.1}
  \textstyle
   \big|
    \big\{
          n \bmod p : \prod_{i=1}^k(n + h_i) \equiv 0 \bmod p 
    \big\} 
   \big| 
    < p
     \qquad(p \in \PP).
\end{align}
The  {\em  prime  $k$-tuple  conjecture} asserts that if  $\A$  is 
admissible then there are infinitely many integers  $n$  for which 
$|\A(n) \cap \PP| = k$.

{\em Level  of  distribution}  concerns  how evenly the primes are 
distributed among arithmetic progressions.
We say that the primes have level of distribution  $\theta$  if 
for any given  $\epsilon \in (0, \theta)$ and  $A > 0$ one has, 
for all $N > 2$, the bound
\begin{align}
 \label{eq:4.2}
  \sum_{q \le N^{\theta - \epsilon}}
   \max_{(q,a) = 1}
    \bigg|
          \psi(N;q,a) - \frac{\psi(N)}{\phi(q)}
    \bigg|
  \ll_{\epsilon,A} \frac{N}{(\log N)^A}.
\end{align}
The celebrated Bombieri--Vinogradov theorem 
\cite[Th\'eor\`eme 17]{BOM} implies  that the primes have level of 
distribution $\theta = \frac{1}{2}$, and the 
{\em Elliott--Halberstam conjecture} \cite{EH,FG} asserts that the 
primes have level of distribution $\theta = 1$.

Next, fix an integer $k \ge 2$ and a number $\eta  \in [0,1)$, and 
for  any  fixed  compactly  supported  square-integrable  function 
$F : [0,\infty+) \to \RR$, define the functionals
\[
 I_k(F) 
  = \int_{[0,\infty+)^k} 
     F(t_1,\ldots,t_k)^2 \, \dd{t_1}\ldots \dd{t_k}
\]
and (for $i = 1,\ldots,k$), 
\[
 J_{i,1 - \eta}(F) 
  = \int_{(1 - \eta)\cdot \mathcal{R}_{k-1}}
     \!\!\!
      \bigg(
       \int_{0}^{\infty} 
        F(t_1,\ldots,t_k) \, \dd{t_i}
      \bigg)^{2} 
       \dd{t_1}\ldots \dd{t_{i-1}} \dd{t_{i+1}} \ldots \dd{t_k},
\]
$(1 - \eta)\cdot\mathcal{R}_{k-1}$ being the simplex
\[
 (1 - \eta)\cdot\mathcal{R}_{k-1} 
  = 
   \{ (t_1,\ldots,t_k) \in [0,1]^{k-1} : 
       t_1 + \cdots + t_k \le 1 - \eta \}.
\]
Define $M_{k,\eta}$ to be the supremum 
\begin{align}
 \label{eq:4.3}
 M_{k,\eta} 
  = 
    \sup
     \frac{\sum_{i=1}^k J_{i,1 - \eta}(F)}{I_k(F)} 
\end{align}
over  square-integrable  functions  $F$  that are supported on the 
simplex
\[
 (1 + \eta)\cdot\mathcal{R}_k 
  = 
   \{ (t_1,\ldots,t_k) \in [0,1]^k : 
       t_1 + \cdots + t_k \le 1 + \eta \},
\]
and are not identically zero.
Maynard  \cite[Proposition 4.2]{MAY}  has shown that for any given 
$m$ there are infinitely many $n$ with 
$|\A(n) \cap \PP| \ge m + 1$, provided 
$M_k  =  M_{k,0}  >  2\theta^{-1}m$, where $\theta$ is 
an admissible level of distribution for $\PP$.
By \cite[Proposition 4.3]{MAY} one has  $M_5 > 2$,  $M_{105} > 4$, 
and  that  $M_k > \log k - \log\log k - 2$  for  all  sufficiently 
large $k$.
A recent  Polymath  project  \cite[Theorem 3.9]{POLY}  has refined 
these bounds as follows:
\begin{align}
 \label{eq:4.4}
 M_{54} > 4,            \quad 
 M_k \ge \log{k} - c,
\end{align}
for some absolute constant $c$. 
Moreover,  the   Polymath   project  has  refined  the  method  of 
\cite{MAY}   slightly,   allowing  one  to  reduce  the  condition 
$M_k > 2\theta^{-1}m$  
to  
$M_{k,\eta} > 2\theta^{-1}m$ 
for some 
$0 \le \eta \le \theta^{-1} - 1$.
They  have  also  \cite[Theorem 3.13]{POLY}   produced  the  bound
\begin{align}
 \label{eq:4.5}
  M_{50,1/25} > 4.
\end{align}

Therefore, if  $\A(x) = \{x + h_i\}_{i=1}^{k}$  is any  admissible 
$k$-tuple, then for infinitely many $n$ we have 
$|\A(n) \cap \PP| \ge m + 1$, provided $k \ge 50$ in the case 
$m = 1$, and $k \ge \e^{4m + c}$ in general.
On  the   Elliott--Halberstam   conjecture  this  holds   provided
$k  \ge  5$  in the case  $m  =  1$,  and  $k \ge \e^{2m + c}$  in 
general. 

The key to  extending Maynard--Tao in the way we require  involves 
an extension of \eqref{eq:4.2}  in  which the moduli  $q$  are all 
multiples of an  integer  $q_0$, which may be as  large as a small 
power of $N$, but all of whose prime factors are relatively small.
This   extension  of  Bombieri--Vinogradov  in   turn  requires  a 
zero free region for the  corresponding  Dirichlet  $L$-functions, 
given by the following lemma.

\begin{lemma}
 \label{lem:4.1}
Let $T \ge 3$ and let $P \ge T^{1/\log_2 T}$.
Among  all  primitive  characters  $\chi \bmod q$  to  moduli  $q$ 
satisfying $q \le T$ and $\gp{q} \le P$,  there is at most one for 
which $L(s,\chi)$ has a zero in the region 
\begin{align}
 \label{eq:4.6} 
  \Re(s) > 1 - \frac{c}{\log P}, \quad 
  |\Im(s)| \le \exp\br{\log P/\sqrt{\log T}}.
\end{align}
where $c$  is a  \textup{(}sufficiently  small\textup{)}  positive 
absolute constant.
If such a character  $\chi \bmod q$  exists, then  $\chi$ is real, 
$L(s,\chi)$  has just one zero in the region \eqref{eq:4.6}, which 
is real and simple, and 
\begin{align}
 \label{eq:4.7}
  \gp{q} \gg \log q \gg \log_2 T.
\end{align}
\end{lemma}
\begin{proof}[Proof outline]
Lemma \ref{lem:4.1} follows from Chang's bound \cite{CHA} for 
character sums to smooth moduli; the argument is somewhat standard 
and so we only give an outline of the proof.

If $\chi \bmod q$ is real and primitive then  $q$ is squarefree up  
to a factor of at most $4$, so 
$\log q \ll \sum_{p \le \gp{q}} \log p \ll \gp{q}$  by Chebyshev's 
bound.
If $\beta$ is any real zero of $L(s,\chi)$ then  
$1 - \beta \gg 1/(\sqrt{q}(\log q)^2)$ \cite[\S14, (12)]{DAV}.
Hence \eqref{eq:4.7}.

If $\chi \bmod q$ is primitive and 
$\kappa\br{\log \gp{q} + \log q/\log_2 q} < \log u < \log q$, 
$\kappa$ a sufficiently  large absolute constant, then a result of 
Chang \cite[Theorem 5]{CHA} yields  
$\sum_{n \le u} \chi(n) \ll u\e^{-\sqrt{\log u}}$.
If  $\gp{q} \le P$,  where  $P \ge q^{1/\log_2 q}$,  we can deduce 
that   $L(\sigma + it, \chi)   \ll   (|t| + 1)P^{\eta}\log P$  for  
$\sigma > 1 - \eta/(2\kappa)$, where 
$0 < \eta \le 1/(2\sqrt{\log q})$.
We do this by writing 
$
 L(s,\chi) 
   = s\int_1^{\infty} u^{-s - 1}( \sum_{n \le u} \chi(n) ) \dd{u},
$
using   Chang's  bound   for  $u$  in  an  applicable  range,  the 
Polya--Vinogradov  bound for larger  $u$  and a  trivial bound for 
smaller $u$.

Under the additional assumption $\eta  \gg  (\log_2 P)/\log P$, we 
can then show by standard calculations (see 
\cite[Lemmas 10--12]{IWA} for instance) that $L(\sigma + it,\chi)$ 
has no zeros for $\sigma  >  1 -  c_1\eta/\log((|t| + 1)P^{\eta})$  
if  $\chi$  is   complex,  and  at  most  one  zero  this  region, 
necessarily real and simple, if $\chi$ is real.
Moreover, we  can  show  that  for  any  distinct  real  primitive 
characters  $\chi_1 \bmod q_1$  and  $\chi_2 \bmod q_2$  (possibly 
$q_1 = q_2$), if $\gp{q} \le P$, where  $P \ge q^{1/\log_2 q}$ and 
$q = [q_1,q_2]$, and if $\beta_1$ and  $\beta_2$ are real zeros of 
$L(s,\chi_1)$ and $L(s,\chi_2)$, then 
$\min(\beta_1,\beta_2) \le 1 - c_2/\log P$.
(Here, $c_1$ and $c_2$ are a constants that are sufficiently small  
in terms of $\kappa$.)
\end{proof}

We fix an absolute  constant  $c$  as  in  Lemma \ref{lem:4.1} and 
define
\begin{align}
 \label{eq:4.8}
  \Z[T] = P^+(q)
\end{align}
if such an exceptional modulus $q$ exists, and $\Z[T]=1$ if no 
such modulus exists.
For future reference, note that the bound \eqref{eq:4.7} implies 
that, regardless of whether or not such a modulus exists, we have
\begin{align}
 \label{eq:4.9}
  \frac{\Z[T]}{\phi(\Z[T])} = 1 + \Oh[\frac{1}{\log_2 T}].
\end{align}

\begin{theorem}[Modified Bombieri--Vinogradov theorem]
 \label{thm:4.2}
Let $N > 2$.
Fix  any $C > 0$ and  $\theta  =  1/2  - \delta  \in  (0,1/2)$.
Fix  $\epsilon  >  0$  and  suppose $q_0$ is a squarefree  integer 
satisfying $q_0 < N^{\epsilon}$ and 
$\gp{q_0} < N^{\epsilon/\log_2 N}$.
If  $\epsilon  =  \epsilon(C,\delta,c)$ is  sufficiently  small in 
terms of $C$, $\delta$ and the constant $c$ in Lemma 
\eqref{lem:4.1}, then with $\Z[N^{2\epsilon}]$ as in 
\eqref{eq:4.8} we have
\begin{align}
 \label{eq:4.10}
  \sums[q \, < \, N^{\theta}]
       [q_0 \, \mid \,\, q \hspace{8pt}]
       [(q, \, {\Z[N^{2\epsilon}]}) \, = \, 1]
   \max_{(q,a) = 1}
    \left| \psi(N;q,a) - \frac{\psi(N)}{\phi(q)} \right|
     \ll_{\delta,C}
      \frac{N}{\phi(q_0)(\log N)^C}.
\end{align}
\end{theorem}
\begin{proof} 
The result follows from  standard zero  density estimates combined 
with  the  zero free  region  for smooth  moduli  given  in  Lemma 
\ref{lem:4.1}.
We assume that  $(q_0, \Z[N^{2\epsilon}]) = 1$,  for otherwise the 
result is trivial. 
First, we rewrite $\psi(N;q,a)$ as 
$\phi(q)^{-1} \sum_{\chi} \psi(N,\chi) \bar{\chi}(a)$, where 
$\psi(N,\chi) = \sum_{n \le N} \chi(n) \Lambda(n)$. 
Next,  we  replace   $\psi(N,\chi)$  with  $\psi(N,\chi')$,  where 
$\chi'$ is the primitive character that induces $\chi$. 
The error in making this change is at most
\[
  \sum_{q < N^{\theta}} \frac{1}{\phi(q)}
   \hspace{3pt}
    \sum_{\chi \bmod q}
     \hspace{3pt}
      \sums[n \le N][(n,q) \ne 1] \Lambda(n)
       \ll N^{\theta} (\log{N})^2,
\]
which is acceptable.
Since   $\psi(N,\chi'_0)  =  \psi(N)$   holds  for  the  principal 
character $\chi_0 \bmod q$, we need only bound
\begin{align}
 \label{eq:4.11}
  \sums[q \, < \, N^{\theta}]
       [q_0 \, \mid \,\, q \hspace{8pt}]
       [(q, \, {\Z[N^{2\epsilon}]}) \, = \, 1] 
   \frac{1}{\phi(q)}
    \hspace{3pt}
     \sums[\chi \bmod{q}][\chi \ne \chi_0]    
     |\psi(N,\chi')|.
\end{align}
For  nonprincipal   characters   $\chi$,   the   explicit  formula 
\cite[\S19, (13)--(14)]{DAV} yields
\begin{align}
 \label{eq:4.12}
 |\psi(N;\chi)|
   \le 
    \sum_{|\rho| < N^{1/2}}   
     \frac{N^{\Re(\rho)}}{|\rho|}
    + \Oh[N^{1/2}(\log{qN})^2],
\end{align}
where  the sum is over nontrivial zeros of  $L(s,\chi)$  with real 
part at least $1/2$. 
The error term here makes a negligible contribution.

We substitute  \eqref{eq:4.12}  into  \eqref{eq:4.11}, and rewrite 
the  summation  in  terms  of  the  moduli  $q'$  of the primitive 
characters that are present. 
Thus, we need to bound
\begin{align*}
  & 
    \sum_{q' \le N^{\theta}} 
     \hspace{3pt}
      \sumss['][\chi \bmod q'] 
       \hspace{3pt}
        \sum_{|\rho| < N^{1/2}} \frac{N^{\Re(\rho)}}{|\rho|}
         \sums[q < N^{\theta}][{[q_0,q']} \,\mid \, q]
              [(q, \, {\Z[N^{2\epsilon}]}) \, = \, 1]    
           \frac{1}{\phi(q)} 
  \\
  & \hspace{30pt}
     \ll \frac{\log N}{\phi(q_0)}
      \sum_{a \, \mid \, q_0}
       \sums[b < N^{\theta}/a]
            [(b,\, q_0{\Z[N^{2\epsilon}]}) \, = \, 1] 
        \frac{1}{\phi(b)}
         \hspace{3pt}
          \sumss['][\chi \bmod ab]
           \hspace{3pt}
            \sum_{|\rho| < N^{1/2}} \frac{N^{\Re(\rho)}}{|\rho|}.
\end{align*}
Here we have written  $q' = ab$  with $a \mid q_0$ and $(b,q_0)=1$
(we are supposing $q_0$  is squarefree);  we use $\sumsstxt[']$ to 
denote a sum restricted to primitive characters.

We cover the sum over $a$ and $b$  with  $\Oh[(\log{N})^2]$ dyadic 
ranges, and the sum over  zeros with  $\Oh[(\log N)^2]$  sums over 
zeros that satisfy
\[
 \Re(\rho)   \in I_m = [1 - m/\log N, 1 - (m - 1)/\log N], \quad
 |\Im(\rho)| \in J_n = [n - 1, 2n],
\]
where $n$ runs over powers of $2$.
Hence, we are left to bound
\begin{align}
 \label{eq:4.13}
  \begin{split}
  &
    \frac{(\log N)^5}{\phi(q_0)}
     \sup_{
           \substack{
                     2m < \log N  \\ 
                     2n < N^{1/2} \\ 
                      A < q_0, \, AB < N^{\theta}
                    }
          }
      \hspace{3pt} 
       \sums[A \le a < 2A]
            [B \le b < 2B]
            [a \, \mid \, q_0, \, 
               (b, \, q_0{\Z[N^{2\epsilon}]}) \, = \, 1]
        \frac{1}{\phi(b)}
         \hspace{6pt} 
          \sumss['][\chi \bmod{ab}]  
           \hspace{3pt} 
            \sums[\Re(\rho)  \, \in \, I_m]
                 [|\Im(\rho)| \, \in \, J_n]
             \frac{N^{\Re(\rho)}}{|\rho|} 
  \\
  & \hspace{75pt} 
     \ll 
      \frac{N(\log N)^6}{\phi(q_0)}
       \sup_{
             \substack{
                       m < \log N  \\ 
                       n < N^{1/2} \\ 
                       A < q_0, \, AB < N^{\theta}
                      }
            }
        \frac{\e^{-m}}{nB}
         N^*\Big( 1 - \frac{m}{\log N}, A, B, n \Big),
  \end{split}
\end{align}
where
\[
 N^*(\sigma,A,B,T)
  =
   \sums[A \le a < 2A][a \, \mid \, q_0]
    \sums[B \le b < 2B]
         [(b,\, q_0 {\Z[N^{2\epsilon}]}) \,= \, 1] 
     \hspace{3pt}
      \sumss['][\chi \bmod ab] 
       \hspace{3pt}
        \sums[|\Im(\rho)| \le T][\Re(\rho) \ge \sigma] 1.
\]

We first consider the range $m \ge C'\log_2 N$ where 
$C' = (C + 15)/\delta$. 
Montgomery's estimate \cite[Theorem 12.2]{MON} shows that
\begin{align}
 \label{eq:4.14}
  N^*(\sigma, A, B, T) 
   \ll (A^2 B^2 T)^{3(1 - \sigma)/(2 - \sigma)} (\log(ABT))^9.
\end{align}
For  $1/2  \le  \sigma  \le 1$,  we  have  $1/(2 - \sigma) \le 1$, 
$6(1 - \sigma)/(2 - \sigma) \le 1 + 2(1 - \sigma)$ and 
$3(1 - \sigma)/(2 - \sigma) \le 1$. 
For $4 \epsilon \le \delta$ we have 
$
 \log(A^6 B^2) 
  \le \log{ N^{2\theta + 4\epsilon} }
   \le (1 - \delta)\log N.
$
Thus, \eqref{eq:4.14} implies
\[
  N^*\Big(1 - \frac{m}{\log N}, A, B, n \Big)
      \ll (\log N)^9 nB\exp(m(1 - \delta)).
\]
After   using  this  bound,   we   see   that   the   supremum  in 
\eqref{eq:4.13}, when restricted to  $m  \ge  C'\log_2 N$,  occurs 
when  $n = 1$,  $m = C'\log_2 N$, $A = q_0$, $B = \log N$, and the 
overall  contribution is  $\ll \phi(q_0)^{-1} N(\log N)^{-C}$,  as 
required.

We now consider the range $m \le C'\log_2 N$. 
In this region \eqref{eq:4.14} implies
\[
  N^*\Big(1 - \frac{m}{\log N}, A, B, n \Big)
   \ll 
   (\log N)^9 n^{1/2} B^{1/2}  
     \exp\Big( 6m \frac{\log A}{\log N} \Big).
\]
After  applying  this  bound,  we  see   the  supremum  occurs  at 
$m = 0$  (since  $A^6 \le N$),  and then it is easy to see that if 
either  $n \ge (\log N)^{2C'}$  or $B \ge (\log N)^{2C'}$ then the 
bound is acceptable. 
We therefore restrict our attention to $n, B < (\log N)^{2C'}$. 

By  Lemma  \ref{lem:4.1},  if  $\chi \bmod q$  is  primitive  with 
$q < N^{2\epsilon}$ and 
$\gp{q}  <  N^{2\epsilon/\log_2 N^{2\epsilon}}$,  then $L(s,\chi)$ 
has no zeros in the region
\[
  \Re(s) > 1 - c\frac{\log_2 N^{2\epsilon}}{\log N^{2\epsilon}},         
   \quad
   |\Im(s)| 
     \le 
      \exp\br{\sqrt{\log N^{2\epsilon}}/\log_2 N^{2\epsilon}},
\]
unless $(q,\Z[N^{2\epsilon}]) \ne 1$.
If  $\epsilon$ is sufficiently small in terms of $C$, $\delta$ and  
$c$,  then  this  region  covers  the  range   $m \le C'\log_2 N$, 
$n \le (\log N)^{2C'}$.
We are supposing that $q_0 < N^{\epsilon}$, 
$\gp{q_0} < N^{\epsilon/\log_2 N}$  and  $B < (\log{N})^{2C'}$, so 
for all remaining moduli $q' = ab$ we certainly have 
$q' < N^{2\epsilon}$ and 
$\gp{q'} < N^{2\epsilon/\log_2 N^{2\epsilon}}$.
Our assumptions also imply that $(q',\Z[N^{2\epsilon}]) = 1$.
\end{proof}

\begin{theorem}
 \label{thm:4.3}
Let $m$, $k$ and $\epsilon = \epsilon(k)$ be fixed.
If $k$ is a sufficiently large multiple of $(8m+1)(8m^2+8m)$ and 
$\epsilon$ is sufficiently  small, there is some  
$N(m,k,\epsilon)$  such that the following holds for all 
$N \ge N(m,k,\epsilon)$.
With $\Z[N^{4\epsilon}]$ given by \eqref{eq:4.8}, let
\begin{align}
 \label{eq:4.15}
  w = \epsilon\log N   
  \quad \text{and} \quad 
    W = \prods[p \le w]
              [\hspace{8pt}p \, \nmid\, {\Z[N^{4\epsilon}]}] p.
\end{align} 
Let $\A  =  \{h_1, \ldots, h_k\}$  be an admissible $k$-tuple such 
that
\begin{align}
 \label{eq:4.16}
  0 \le h_1,\ldots,h_k \le N
\end{align}
and
\begin{align}
 \label{eq:4.17}
  \textstyle  
 p \, \Big\vert  \prod_{1 \le i < j \le k}(h_j - h_i) 
    \implies p \le w.
\end{align}
Let 
\begin{align}
 \label{eq:4.18}
 \A   = \A_1^{(1)} \cup \cdots \cup \A_{8m+1}^{(1)}   
      = \A_1^{(2)} \cup \cdots \cup \A_{8m^2+8m}^{(2)}
\end{align}
be two partitions of $\A$ into $8m + 1$ and $8m^2 + 8m$ sets of 
equal size respectively.
Finally, let $b$ be an integer such that  
\begin{align}
 \label{eq:4.19}
 \textstyle
  \br{ \prod_{i = 1}^k (b + h_i),W } = 1.
\end{align}
\textup{(}i\textup{)} 
There is some  $n_1 \in (N,2N]$ with  $n_1 \equiv b \bmod W$, and 
some set of $m+1$ distinct indices 
$
 \{i_1^{(1)},\ldots,i_{m + 1}^{(1)}\}
  \subseteq 
   \{1,\ldots,8m + 1\},
$ 
such that 
\begin{align}
 \label{eq:4.20}
 |\A_i^{(1)}(n_1) \cap \PP| = 1
   \quad 
    \text{for all $i \in \{i_1^{(1)},\ldots,i_{m+1}^{(1)}\}$}.
\end{align}
\textup{(}ii\textup{)}  
There is some $n_2 \in (N,2N]$ with  $n_2 \equiv b \bmod W$, and 
some set of $m + 1$ distinct indices 
$
 \{i_1^{(2)},\ldots,i_{m + 1}^{(2)}\}
  \subseteq 
   \{1,\ldots,8m^2+8m\},
$
such that 
\begin{align}
 \label{eq:4.21}
  \begin{split}
 |\A_i^{(2)}(n_2) \cap \PP| 
  & = 1
   \quad 
    \text{for all $i \in \{i_1^{(2)},\ldots,i_{m + 1}^{(2)}\}$},  
 \\
  |\A_i^{(2)}(n_2) \cap \PP| 
  & \le 1
     \quad 
      \text{for all $i_1^{(2)}<i<i_{m+1}^{(2)}$}.
  \end{split}
\end{align}
\end{theorem}

If we fix $m$, $k$ and $\eta \in [0,1)$ with   
$M_{k,\eta} - 4m > 0$ (where $M_{k,\eta}$ is as in 
\eqref{eq:4.3}), and if we assume the remaining hypotheses of 
Theorem \ref{thm:4.3} hold (disregarding \eqref{eq:4.18}), then we 
can show, for all $N \ge N(m,k,\epsilon)$, that 
$|\A(n) \cap \PP| \ge m + 1$ for at least one $n \in (N,2N]$ with 
$n \equiv b \bmod W$.
This follows from an essentially identical argument to that 
presented in \cite{MAY, POLY}, although there are two differences 
in our setting that potentially affect the argument. 
Namely, $w$ is considerably larger here than in \cite{MAY} or 
\cite{POLY} (we take $w = \epsilon\log N$ instead of 
$w = \log_3 N$), and the elements of $\A$ here may vary with $N$.

However, this actually only leads to weaker versions of 
Theorems \ref{thm:1.1} and \ref{thm:1.3}, for instance (cf.\  
\eqref{eq:4.4}, \eqref{eq:4.5}) with $k = 50$ instead of $k = 9$ 
in our main theorem (Theorem \ref{thm:1.1}), which is concerned 
with the case $m = 1$ of Theorem \ref{thm:4.3}.
The proof of Theorem \ref{thm:4.3}, given in Section 
\ref{subsec:4.2}, does not require such refined estimates as in 
\cite{MAY, POLY}, but does require an additional sieve upper 
bound, whose use had been considered by the authors of 
\cite{POLY}.

We remark that with  more  significant  modifications to  the 
argument presented in \cite{MAY,POLY}, it is in principle possible  
to  remove the requirement \eqref{eq:4.17} from  the statement of 
Theorem \ref{thm:4.3}.
We do not consider this here.

\subsection{Key estimates}
 \label{subsec:4.2}

Throughout this section (including Lemmas \ref{lem:4.4} -- 
\ref{lem:4.6}): 
$k$ is fixed;
$\delta > 0$ and $\epsilon > 0$ are fixed and satisfy 
$2\delta + 2\epsilon < \frac{1}{2}$ (as well as 
$\delta > 2\epsilon$ in Lemma \ref{lem:4.6} (iii));
$N$ is to be thought of as tending to infinity, hence is 
sufficiently large in terms of any fixed quantity; 
implicit constants may depend on any fixed quantity (though our 
notation will not indicate this explicitly); 
$\Z[N^{4\epsilon}]$ is given by \eqref{eq:4.8};
$w$, $W$, $\A = \{h_1,\ldots,h_k\}$ and $b$ are as in 
\eqref{eq:4.15} -- \eqref{eq:4.19}.
(Note that by the prime number theorem, 
$W < N^{2\epsilon}$, hence $N^{2\delta}W < N^{\theta}$ where 
$\theta < 1/2$, and likewise if $\delta \ge 2\epsilon$ then 
$N^{1/2 - \delta}W < N^{\theta}$ where $\theta < 1/2$.)

Also, $\lambda_{d_1,\ldots,d_k}$ are sieve weights given by 
\begin{align}
 \label{eq:4.22}
  \begin{split}
  \lambda_{d_1,\ldots,d_k}
 = \begin{cases}
    \big(
     \prod_{i=1}^k \mu(d_i)
    \big)
     \sum_{j=1}^J 
      \prod_{\ell = 1}^k 
       F_{\ell,j}
        \big(
         \frac{\log{d_\ell}}{\log{N}}
        \big) 
    \quad & \text{if $(d_1\cdots d_k, Z_{N^{4\epsilon}}) = 1$,} 
     \\
    0 & \text{otherwise,}
   \end{cases}
  \end{split}
\end{align}
for some fixed $J$ and fixed smooth nonnegative compactly 
supported functions 
$
F_{\ell,j} : [0,\infty+) \to \mathbb{R}
$ 
that are not identically $0$ and that satisfy the support 
restriction
\[
 \sup \big\{ \textstyle
       \sum_{\ell = 1}^k t_\ell : 
        \prod_{\ell = 1}^k F_{\ell,j}(t_\ell) \ne 0
      \big\}
  \le \delta,
\]
for each $j \in \{1,\ldots,J\}$.
This support condition implies $\lambda_{d_1,\ldots,d_k}$ is 
supported on $d_i$ with $\prod_{i = 1}^k d_i \le N^{\delta}$.
The fact that $J$, $F_{\ell,j}$ are fixed means we have the bound 
$\lambda_{d_1,\ldots,d_k} \ll 1$ uniformly in the $d_i$. 
To ease notation we put  
\[
 F(t_1,\ldots,t_k) 
  = 
   \sum_{j = 1}^J 
    \prod_{\ell = 1}^k 
     F'_{\ell,j}(t_\ell),
\]
$F'_{\ell,j}$ denoting the derivative of $F_{\ell,j}$, and we 
assume that we have chosen the $F_{\ell,j}$ such that $F$ is 
symmetric.
Also, we put
\[
 B = \frac{\phi(W)}{W}\log N.
\]

\begin{lemma}
 \label{lem:4.4}
If  
$
 F_1,\ldots,F_k, G_1,\ldots,G_k : [0,\infty+) \to \RR
$ 
are fixed smooth compactly supported functions, then
\begin{align*}
 \sumss['][d_1,\ldots,d_k][d_1',\ldots,d_k']
  \hspace{5pt}
   \prod_{j=1}^k
           \frac{\mu(d_j)\mu(d_j')} 
                {[d_j,d_j']}
            F_j\bigg(\frac{\log d_j}{\log N}\bigg)
            G_j\bigg(\frac{\log d_j'}{\log N}\bigg)
  = 
    (c + \oh[1])B^{-k},
\end{align*}
where  $\sumsstxt[']$  denotes summation with the restriction that 
$[d_1,d_1'],\ldots, [d_k,d_k'], W\Z[N^{4\epsilon}]$  are  pairwise 
coprime, and  
\[
 c = \prod_{j=1}^k \int_0^{\infty} F_j'(t_j)G_j'(t_j) \dd{t_j}.  
\]
The same holds if the  denominators  $[d_j,d_j']$  are replaced by 
$\phi([d_j,d_j'])$.
\end{lemma}

\begin{proof} 
This is \cite[Lemma 4.1]{POLY} combined with the fact that, by 
\eqref{eq:4.9}, 
\[
(\Z[N^{4\epsilon}]/\phi(\Z[N^{4\epsilon}]))^k = 1 + \oh[1].
\] 
\end{proof}

We may now prove the main estimates of the Maynard--Tao sieve 
method.
To state the estimates we define
\begin{align*}
 I_k(F) 
  & = 
   \int_0^{\infty}
    \hspace{-6pt} \cdots 
     \int_0^{\infty}
      F(t_1,\ldots,t_k)^2 \dd{t_1}\ldots\dd{t_k}, 
 \\
 J_k(F) 
  & =
   \int_0^{\infty}
    \hspace{-6pt} \cdots 
     \int_0^{\infty}
    \bigg(
     \int_0^{\infty} F(t_1,\ldots,t_k) \dd{t_k}
    \bigg)^2
     \dd{t_1}\ldots\dd{t_{k-1}} 
 \intertext{and} 
 L_k(F) 
  & =
   \int_0^{\infty}
    \hspace{-6pt} \cdots 
     \int_0^{\infty}
    \bigg(
     \int_0^{\infty}
      \int_0^{\infty}
       F(t_1,\ldots,t_k) \dd{t_{k-1}} \dd{t_k}
    \bigg)^2
     \dd{t_1}\ldots\dd{t_{k-2}}.
\end{align*}

\begin{lemma}
 \label{lem:4.5}
\textup{(}i\textup{)}
We have
\begin{align*}
 \sums[N < n \le 2N][n \equiv b \bmod W]
  \bigg(
   \sums[d_1,\ldots,d_k][d_i \mid n + h_i]
    \lambda_{d_1,\ldots,d_k}
  \bigg)^2
   =
    (1 + \oh[1])\frac{N}{W}B^{-k}I_k(F).
\end{align*}
\textup{(}ii\textup{)}
For each $j \in \{1,\ldots,k\}$, we have
\begin{align*}
 \sums[N < n \le 2N][n \equiv b \bmod W]
  \ind{\PP}(n + h_j)
   \bigg(
    \sums[d_1,\ldots,d_k][d_i \mid n + h_i]
     \lambda_{d_1,\ldots,d_k}
   \bigg)^2
   =
    (1 + \oh[1])\frac{N}{W}B^{-k}J_k(F).
\end{align*}
\textup{(}iii\textup{)}
For each pair $j,\ell \in \{1,\ldots,k\}$, $j \ne \ell$, we have
\begin{align*}
 \sums[N < n \le 2N][n \equiv b \bmod W]
  \ind{\PP}(n + h_j)
   \ind{\PP}(n + h_{\ell})
    \bigg(
     \sums[d_1,\ldots,d_k][d_i \mid n + h_i]
      \lambda_{d_1,\ldots,d_k}
    \bigg)^2
   \le
    (4 + \Oh[\delta])\frac{N}{W}B^{-k}L_k(F).
\end{align*}
\end{lemma}

\begin{proof}
(i) We expand the square and swap the order of summation to obtain
\[
 \sums[N <  n \le 2N][n \equiv b \bmod W]
  \bigg(
   \sums[d_1,\ldots,d_k][d_i \mid n + h_i]
    \lambda_{d_1,\ldots,d_k}
  \bigg)^2
 = \sums[d_1,\ldots,d_k][d_1',\ldots,d_k']
    \lambda_{d_1,\ldots,d_k}\lambda_{d_1',\ldots,d_k'}
     \sums[N < n \le 2N]
          [n \equiv b \bmod W]
          [{[d_i,d_i']} \mid n + h_i] 
          1.
\]
By our choice of $b$, there is no contribution to the inner sum 
unless all the $d_i$ and $d_i'$ are coprime to $W$. 
By the restriction on the support of $\lambda_{d_1,\ldots,d_k}$, 
there is no contribution unless all the $d_i, d_i'$ are coprime to 
$\Z[N^{4\epsilon}]$.
Since $p$ does not divide $\prod_{i \ne j}(h_i - h_j)$ unless 
$p \le w$, we see that there is no contribution unless all of 
$[d_1,d_1'],\ldots,[d_k,d_k']$ are pairwise coprime. 
If all these conditions are satisfied then the inner sum is equal 
to
\[
 \frac{N}{W \prod_{i = 1}^k [d_i,d_i']} + \Oh[1].
\]
Since $\lambda_{d_1,\ldots,d_k} \ll 1$ and is supported 
on $\prod_{i=1}^k d_i \le N^{\delta}$, we see that the error 
term trivially contributes $\Oh[N^{2\delta + \oh(1)}]$, which is 
negligible. 

Expanding $\lambda_{d_1,\ldots,d_k}$ using the definition 
\eqref{eq:4.22}, we see that the main term contributes
\[
 \frac{N}{W}
  \sum_{j = 1}^J
   \sum_{j' = 1}^J
    \hspace{5pt}
     \sumss['][d_1,\ldots,d_k][d_1',\ldots,d_k']
      \hspace{5pt}
       \prod_{\ell = 1}^k
        \frac{\mu(d_\ell)\mu(d_\ell')}
             {[d_\ell,d_\ell']}
         F_{\ell,j}\bigg(
                    \frac{\log{d_\ell}}{\log N}
                   \bigg)
         F_{\ell,j'}\bigg(
                     \frac{\log{d_\ell'}}{\log N}
                    \bigg),
\]
where $\sumsstxt[']$ signifies pairwise coprimality of  
$[d_1,d_1'],\ldots, [d_k,d_k'], W\Z[N^{4\epsilon}]$.
The inner sum can be estimated by Lemma \ref{lem:4.4}, which gives 
the result.

(ii) The argument here is similar. 
For ease of notation we will consider $j = k$, the other cases 
being entirely analogous. 
There is no contribution to the sum unless $d_k = 1$. 
With this restriction, we expand the square and swap the order of 
summation to obtain
\begin{align*}
 & \sums[N < n \le 2N][n \equiv b \bmod W]
    \ind{\PP}(n + h_j)
     \bigg(
      \sums[d_1,\ldots,d_{k-1}][d_i \mid n + h_i] 
       \lambda_{d_1,\ldots,d_{k-1},1}
     \bigg)^2
 \\
 & \hspace{60pt} =
  \sums[d_1,\ldots,d_{k-1}]
       [d_1',\ldots, d_{k-1}']
        \lambda_{d_1,\ldots,d_{k-1},1}
         \lambda_{d_1',\ldots,d_{k-1}',1}
    \sums[N < n \le 2N]
         [n \equiv b \bmod W]
         [{[d_i,d_i'] \mid n + h_i}]
          \ind{\PP}(n + h_j).
\end{align*}
As in (i), we may assume pairwise coprimality of  
$[d_1,d_1'],\ldots,[d_k,d_k'],W\Z[N^{4\epsilon}]$, in which case 
the inner sum is equal to 
\[
 \frac{\pi(2N + h_j) - \pi(N + h_j)}
      {\phi(W) \prod_{i = 1}^k \phi([d_i,d_i'])}
  + \Oh[E(N;{[d_1,d_1'] \cdots [d_k,d_k']}W)],
\]
where
\[
 E(N;q) 
  = 
   \max_{\substack{(a,q) = 1 \\ h \in \A}}
    \bigg|
     \pi(2N + h;q,a) - \pi(N + h;q,a) 
    - \frac{\pi(2N + h) - \pi(N + h)}{\phi(q)}
    \bigg|.
\]
By the bound $\lambda_{d_1,\ldots,d_k} \ll 1$, the trivial bound 
$E(N;q) \ll 1 + N/\phi(q)$, the Cauchy--Schwarz inequality and 
Theorem \ref{thm:4.2}, the error contributes
\begin{align*}
 & \sumss['][d_1,\ldots,d_k][d_1',\ldots,d_k]
   |\lambda_{d_1,\ldots,d_{k-1},1}\lambda_{d_1',\ldots,d_{k-1}',1}
    |E(N;W[d_1,d_1']\cdots[d_k,d_k'])
 \\
 & \ll 
    \sums[r \le N^{2\delta}][(r,W{\Z[N^{2\epsilon}]) = 1}] 
     \mu(r)^2\tau_{3k}(r) E(N;rW) 
 \\
 & \ll 
    \bigg(
     \sums[r \le N^{2\delta}][(r,W{\Z[N^{2\epsilon}]) = 1}] 
      \mu(r)^2\tau_{3k}(r)^2(1 + N/\phi(rW))
    \bigg)^{1/2}
    \bigg(
     \sums[r \le N^{2\delta}]
          [(r,W{\Z[N^{2\epsilon}]) = 1}] 
           \mu(r)^2E(N;rW)
    \bigg)^{1/2}
 \\
  & \ll 
     \frac{N}{W(\log N)^{2k}}.
\end{align*}

As in (i), expanding $\lambda_{d_1,\ldots,d_k}$ using the 
definition \eqref{eq:4.22} and applying Lemma \ref{lem:4.4} 
to the resulting sums shows that the main term contributes
\[
  (1 + \oh[1])\frac{N}{W}B^{-k}
   \sum_{j = 1}^J
    \sum_{j' = 1}^J 
     F_{k,j}(0)F_{k,j'}(0)
      \prod_{\ell = 1}^{k - 1}
       \int_0^{\infty} 
        F'_{\ell,j}(t_\ell)F'_{\ell,j'}(t_\ell) \dd{t_\ell}.
\]
Noting that the double sum is $J_k(F)$ and that assumed symmetry 
of $F$ means that the expression is independent of 
$j \in \{1,\ldots,k\}$, this gives the result.

(iii) As in (ii), we see there is no contribution unless 
$d_j = d_{\ell} = 1$. 
We first impose this restriction, and then use the sieve upper 
bound
\[
 \ind{\PP}(n + h_\ell)
  \le 
   \bigg(
    \sum_{e \mid n + h_\ell}
     \mu(e)
      G\bigg(
        \frac{\log e}{\log R}
        \bigg)
   \bigg)^2,
\]
for a smooth function $G : [0,\infty+) \to \RR$ supported on 
$[0,1/4 - 2\delta]$, with $G(0) = 1$. 
(The use of such a bound was previously suggested in discussions 
of the Polymath 8b project.)
Thus, we have
\begin{align*}
 & \sums[N < n \le 2N][n \equiv b \bmod W]
    \ind{\PP}(n + h_j)
     \ind{\PP}(n + h_\ell)
      \bigg(
       \sums[d_1,\ldots,d_k][d_i \mid n + h_i]
        \lambda_{d_1,\ldots,d_k}
      \bigg)^2
 \\
 & \hspace{30pt} \le 
    \sums[N < n \le 2N][n \equiv b \bmod W]
     \ind{\PP}(n + h_j)
      \bigg(
       \sum_{e \mid n + h_\ell}
        \mu(e)
         G
          \bigg(
           \frac{\log e}{\log R}
          \bigg)
      \bigg)^2
      \bigg(
       \sums[d_1,\ldots,d_k]
            [d_i \mid n + h_i]
            [d_j = d_\ell = 1]
            \lambda_{d_1,\ldots,d_k}
      \bigg)^2.
\end{align*}
The right-hand side of this expression is now of essentially an 
identical form to that of part (ii), with $F$ replaced by 
$\tilde{F}$, where
\[
 \tilde{F}(t_1,\ldots,t_k)
  =
   G'(t_{\ell})
    \int_0^{\infty} 
     F(t_1,\ldots,t_{\ell-1},u_{\ell},t_{\ell + 1},\ldots,t_k)
      \dd{u_{\ell}}.
\]
(The cases where $j \ge \ell - 1$ are analogous.)
We note that $\tilde{F}$ is supported on $t_1,\ldots,t_k$ such 
that $\sum_{i = 1}^k t_i \le 1/4 - \delta$, by the 
support of $F$ and $G$. 
This means we can still apply Theorem \ref{thm:4.2} as in (ii) 
(since we may restrict to arithmetic progressions modulo $rW$, 
where 
$
 r = [d_1,d_1']\cdots [d_k,d_k'][e_\ell,e_\ell']
  \le N^{1/2 - \delta}
$).
Therefore the same argument as in (ii) gives
\begin{align*}
 & 
  \sums[N < n \le 2N][n \equiv b \bmod W] 
   \ind{\PP}(n + h_j)
    \bigg(
     \sum_{e \mid n + h_{\ell}}
      \mu(e)
       G
        \bigg(
         \frac{\log e}{\log R}
        \bigg)
    \bigg)^2
    \bigg(
     \sums[d_1,\ldots,d_k]
          [d_i \mid n + h_i]
          [d_j = d_{\ell} = 1]
           \lambda_{d_1,\ldots,d_k}
    \bigg)^2
 \\
  & =  
  (1 + \oh[1])\frac{N}{W}B^{-k}
    \int_0^{\infty}
     \hspace{-6pt}\cdots 
      \int_0^{\infty}
       \bigg(\int_0^{\infty} 
        \tilde{F}(t_1,\ldots,t_k) \dd{t_j}
       \bigg)^2
        \dd{t_1}\ldots \dd{t_{j-1}}\dd{t_{j+1}}\ldots \dd{t_k} 
 \\
 & =
    (1 + \oh[1])\frac{N}{W}B^{-k}
     L_k(F)
      \int_0^{\infty} G'(t_\ell)^2 \dd{t_{\ell}}.
\end{align*}
Finally, we take $G(t)$ to be a fixed smooth approximation to 
$1 - t/(1/4 - \delta)$ supported on $1/4 - t$ with 
$G(0) = 1$ and 
$
 \int_0^{\infty} G'(t)^2 \dd{t}
  \le 
   4 + \Oh[\delta].
$
This gives the result.
\end{proof}

\begin{lemma}
 \label{lem:4.6}
Let $0 < \rho < 1$.
Then there is a fixed choice of $J$ and $F_{\ell,j}$ for 
$\ell \in \{1,\ldots,k\}$, $j \in \{1,\ldots,J\}$, with the 
required properties such that
\begin{align*}
 J_k(F) 
  & \ge 
   \br{1 + \Oh[(\log k)^{-1/2}]}
    \bigg(
     \frac{\rho\delta\log k}{k}
    \bigg) 
     I_k(F),
 \\
 L_k(F)
  & \le 
   \br{1 + \Oh[(\log k)^{-1/2}]}
    \bigg(
     \frac{\rho \delta \log k}{k}
    \bigg)^2 
     I_k(F).
\end{align*}
\end{lemma}

\begin{proof}
This follows from the method of \cite[Proposition 4.3]{MAY}. 
The result is trivial if $k$ is bounded, so we assume that $k$ is 
sufficiently large. 
Let $F_k = F_k(t_1,\ldots,t_k)$ be defined by
\begin{align*}
 F_k(t_1,\ldots,t_k)
 & =
  \begin{cases}
   \prod_{i = 1}^k g(kt_i) & \text{if $\sum_{i=1}^k t_i\le 1$,}
   \\
  0 & \text{otherwise,}
  \end{cases}
 \\
 g(t)
  & = 
  \begin{cases}
   1/(1 + At) & \text{if $t\in[0,T]$,}
   \\
   0          & \text{otherwise,}
\end{cases}
 \\
 A & = \log k - 2\log\log k ,
 \\
 T & = (\e^A - 1)/A.
\end{align*}
The proof of \cite[Proposition 4.3]{MAY} shows that
\[
 J_k(F_k) \ge \br{1 + \Oh[(\log k)^{-1/2}]}(\log k)I_k(F_k)/k.
\]
We see that
\begin{align*}
 \bigg(
  \int_0^x g(t) \dd{t}
 \bigg)^2
 & = 
  \min
   \bigg(
    \bigg( \frac{\log(1 + Ax)}{A} \bigg)^2,1
   \bigg)
 \\
  & \le 
   (\log k)
     \min
      \bigg( \frac{x}{1 + Ax}, \frac{T}{1 + AT} \bigg)
  =
  (\log k)
    \int_0^x g(t)^2 \dd{t}
\end{align*}
for any $x \ge 0$. 
Hence
\begin{align*}
 L_k(F_k)
  & =
   \idotsint\limits_{\sum_{i = 1}^{k - 2} t_i \le 1}
    \bigg(
     \prod_{i = 1}^{k - 2} g(k t_i)^2
    \bigg)
 \\
 & \hspace{30pt} \times 
    \bigg(
     \int_0^{1 - \sum_{i = 1}^{k - 2} t_i} g(k t_{k-1})
      \int_0^{1 - \sum_{i = 1}^{k - 1} t_i} g(kt_{k}) 
       \dd{t_k} \dd{t_{k-1}}
    \bigg)^2 
     \dd{t_1} \ldots \dd{t_{k-2}}
 \\
 & \le 
  \bigg(
   \frac{\log k}{k}
  \bigg)^2
   \idotsint\limits_{\sum_{i = 1}^{k - 2} t_i \le 1}
    \bigg(
     \prod_{i = 1}^{k - 2} g(k t_i)^2
    \bigg)
 \\
 & \hspace{30pt} \times 
  \bigg(
   \int_0^{1 - \sum_{i = 1}^{k - 2} t_i} g(kt_{k-1})^2
    \int_0^{1 - \sum_{i = 1}^{k - 1} t_i} g(kt_{k})^2 
     \dd{t_k} \dd{t_{k-1}}
  \bigg) 
   \dd{t_1}\ldots \dd{t_{k-2}}
 \\
 & = 
  \bigg(
   \frac{\log{k}}{k}
  \bigg)^2 
   I_k(F_k).
\end{align*}
By the Stone--Weierstrass theorem we can take $F(t_1,\ldots,t_k)$ 
to be a smooth approximation to 
$
 F_k(\rho \delta t_1,\ldots,\rho \delta t_k)
$ 
such that 
\begin{align*}
 I_k(F) 
  & = 
   (\delta \rho)^k
    \br{1 +\Oh[(\log k)^{-1/2}]}
     I_k(F_k),
 \\
 J_k(F)
  & =
   (\delta \rho)^{k + 1}
    \br{1 + \Oh[(\log k)^{-1/2})]}
     J_k(F_k)
 \intertext{and}
 L_k(F)
  & = 
   (\delta \rho)^{k + 2}
    \br{1 + \Oh[(\log k)^{-1/2})]}
     L_k(F_k).
\end{align*}
This gives the result.
\end{proof}

\begin{proof}[Deduction of Theorem \ref{thm:4.3}]
We first consider part (i). 
We suppose $k$ is a multiple of 
$8m + 1$ and 
\[
 \A = \A_1^{(1)} \cup \cdots \cup \A_{8m+1}^{(1)}
\] 
is a partition of $\A$ into $8m+1$ sets each of size $k/(8m + 1)$. 
We consider
\begin{multline*}
 S 
  = 
   \sum_{N < n \le 2N}
    \bigg(
     \sum_{i = 1}^k 
      \ind{\PP}(n + h_i) 
      - m
      - \sum_{j = 1}^{8m+1}
         \sums[h, h' \in \A_j^{(1)}][h \ne h']
          \ind{\PP}(n + h)
           \ind{\PP}(n + h')
    \bigg)
 \\ 
   \times 
    \bigg(
     \sums[d_1,\ldots,d_k][d_i \mid n+h_i \, \forall i]
      \lambda_{d_1,\ldots,d_k}
    \bigg)^2.
\end{multline*}
We note that if $S > 0$ then there must be at least one $n$ that 
makes a positive contribution to the sum, and this occurs only 
when there exists $m+1$ elements $h'_1,\ldots,h_{m+1}'$ of $\A$ 
each in different subsets $\A_i^{(1)}$ such that $n + h'_j$ is 
prime for all $1 \le j \le m + 1$.
By Lemmas \ref{lem:4.5} and \ref{lem:4.6}, we see that for 
$k > k_0(m,\delta)$, by choosing $\rho < 1$ such that 
$\delta \rho \log k = 2m$ there exists a choice of $F$ such that
\begin{align*}
 S
  & =
   \frac{N}{W}B^{-k}I_k(F)
    \bigg(
     \sum_{i = 1}^k \frac{2m}{k}
       - m
       - 4\sum_{j = 1}^{8m+1}
           \sums[h, h' \in \A_j][h \ne h']
            \frac{(2m)^2}{k^2}
             + \Oh[\delta]
     \bigg)
 \\
  & = 
   \frac{N}{W}B^{-k}I_k(F)
    \bigg(
     \frac{m}{1+8m} + \frac{8 m^2}{k} + \Oh(\delta)
    \bigg).
\end{align*}
Thus, $S > 0$ for $\delta$ sufficiently small, as required.

Part (ii) follows from an essentially identical argument. 
Given a partition
\[
 \A = \A_1^{(2)} \cup \cdots \cup \A_{8m^2 + 8m}^{(2)}
\] 
of $\A$ into equally sized sets, we consider  
\begin{align*}
 & 
 S' = 
 \\
  & 
   \sum_{N < n \le 2N}
    \bigg(
     \sum_{i = 1}^k 
      \ind{\PP}(n + h_i)
       - m
        - (m+1)\sum_{j = 1}^{8m^2+8m}
         \sums[h, h' \in \A_j^{(2)}][h \ne h']
          \ind{\PP}(n + h)
           \ind{\PP}(n + h')
    \bigg)
 \\  
  & \hspace{280pt}
   \times 
    \bigg(
     \sums[d_1,\ldots,d_k][d_i \mid n+h_i \, \forall i]
      \lambda_{d_1,\ldots,d_k}
    \bigg)^2.
\end{align*}
If $n$ makes a positive contribution to $S'$ then we must have 
that the number of indices $j$ for which 
$|\A_j^{(2)}(n) \cap \PP| = 1$ is at least $m + 1 + mr$, where $r$ 
is the number of indices $i$ for which
$|\A_i^{(2)}(n) \cap \PP| > 1$. 
Thus in particular, there must be some set of $m+1$ indices 
$i_1 < \cdots < i_{m + 1}$ for which 
$|\A_i^{(2)}(n) \cap \PP| = 1$ 
for $i = i_1, \ldots,i_{m + 1}$, and 
$|\A_i^{(2)}(n) \cap \PP| = 0$ for $i_1 < i < i_{m + 1}$ and 
$i \ne i_1,\ldots,i_{m + 1}$.
Applying Lemmas \ref{lem:4.5} and \ref{lem:4.6} and choosing 
$\delta \rho \log k = 2m$ as above, we find that $S' > 0$ for 
$\delta$ sufficiently small and $N$ sufficiently large, so such an 
$n$ must exist.
\end{proof}

%%%%%%%%%%%%%%%%%%%%%%%%%%%%%%%%%%%%%%%%%%%%%%%%%%%%%%%%%%%%%%%%%%
%%%%%%%%%%%%%%%%%%%%%%%%%%%  SECTION 5 %%%%%%%%%%%%%%%%%%%%%%%%%%%
%%%%%%%%%%%%%%%%%%%%%%%%%%%%%%%%%%%%%%%%%%%%%%%%%%%%%%%%%%%%%%%%%%

\section{An Erd\H os--Rankin type construction}
 \label{sec:5}

We  give  our   Erd{\H o}s--Rankin   type  construction  in  Lemma 
\ref{lem:5.2}.
We need the following elementary lemma. 
\begin{lemma}
 \label{lem:5.1}
Let $\{h_1,\ldots,h_k\}$  be an admissible  $k$-tuple, let  $S$ be   
a set of integers, and let $\mathscr{P}$  be a set of primes, such 
that for some $x \ge 2$,
\[
 \textstyle
  \{h_1,\ldots,h_k\} 
   \subseteq S 
    \subseteq [0,x^2]
     \qquad \text{and} \qquad 
     |\{p \in \mathscr{P} : p > x\}| > |S| + k.
\]
There is a set of integers  $\{a_p : p \in \mathscr{P}\}$ with the 
property that
\[
 \textstyle
  \{h_1,\ldots,h_k\}
  = S\setminus
      \bigcup_{p \in \mathscr{P}} 
       \{g : g \equiv a_p \bmod p\}.
\]
\end{lemma}
\begin{proof}
First, we observe the following.
Let  $\{ h_1, \ldots, h_k \}$  be  an  admissible  $k$-tuple,  let 
$\mathscr{P}_0  \subseteq  \mathscr{P}$ be sets of primes, and let 
$\{a_p : p \in \mathscr{P}_0\}$ be a set of integers. 
If 
\[
 \textstyle
  \{h_1,\ldots,h_k\}
   = 
    S
     \setminus
      \bigcup_{p \in \mathscr{P}_0} 
       \{g : g \equiv a_p \bmod p\},
\]
then we  can add  integers to  $\{a_p : p \in \mathscr{P}_0\}$  to 
form a set $\{ a_p : p \in \mathscr{P} \}$ such that
\[
 \textstyle
  \{h_1,\ldots,h_k\}
   = 
    S
     \setminus
      \bigcup_{p \in \mathscr{P}} 
       \{g : g \equiv a_p \bmod p\}.
\]
Indeed, since $\{h_1,\ldots,h_k\}$  is admissible, for every prime 
$p$  there is a congruence class  $b_p \bmod p$  for which 
$
 \prod_{i=1}^k(b_p - h_i) \not\equiv 0 \bmod p,
$
so for any $p \in \mathscr{P}\setminus\mathscr{P}_0$ we can choose 
$a_p$ with $a_p \equiv b_p \bmod p$.

Second, we observe that for any given integer $n$, if 
\[
 \textstyle
  \prod_{i=1}^k(n - h_i) \equiv 0 \bmod p
\]
for every prime $p$ in a set  $\mathscr{P}_0$  of  $k + 1$ or more 
distinct primes, then
\[
 n - h_i \equiv 0 \bmod pp'
\]
for some $h_i \in \{h_1,\ldots,h_k\}$ and 
$p,p' \in \mathscr{P}_0$,
so either $n - h_i = 0$ or $|n - h_i| \ge pp'$.
Therefore, if $0 \le n, h_1, \ldots, h_k \le x^2$, 
$n \not\in \{h_1,\ldots,h_k\}$, and  $\mathscr{P}_0$ is any set of 
primes at least $k + 1$ of which are greater than $x$, there must 
be a prime $p \in \mathscr{P}_0$ such that
\[
 \textstyle
  \prod_{i=1}^k(n - h_i) \not\equiv 0 \bmod p.
\]

Now, let $\{h_1,\ldots,h_k\}$ be an admissible $k$-tuple contained 
in  $S  \subseteq [0, x^2]$, and let  $\mathscr{P}$  be any set of 
primes such that  
$|\{p \in \mathscr{P} : p > x\}| \ge |S| + k + 1$.
By our first observation it suffices to show that 
\[
 \textstyle
 \{h_1,\ldots,h_k\}
  = S\setminus
      \bigcup_{p \in \mathscr{P}_0} 
       \{g : g \equiv a_p \bmod p\},
\]
for some $\mathscr{P}_0 \subseteq \mathscr{P}$.
Suppose $n \in S\setminus\{h_1,\ldots,h_k\}$.
By our second observation we may choose a prime 
$p \in \mathscr{P}$ such that 
$
\prod_{i=1}^k(n - h_i) \not\equiv 0 \bmod p.
$
Choose any such prime $p$ and choose any $a_{p}$ with 
$a_{p} \equiv n \bmod p$.
Let  $S_1 = S \setminus \{g : g \equiv a_{p} \bmod p\}$,  so  that 
$n \not\in S_1$, and let 
$\mathscr{P}_1 = \mathscr{P} \setminus \{p\}$.
If $S_1 = \{h_1,\ldots,h_k\}$ then we're done.
Otherwise, we have $\{h_1,\ldots,h_k\} \subsetneq S_1$ and  
$
 |\mathscr{P}_1| 
  = |\mathscr{P}| - 1
   \ge |S| + k
    \ge |S_1| + k + 1.
$
We repeat the above argument as many times as necessary.
\end{proof}

To prove Lemma \ref{lem:5.2} we also need some standard estimates.
First, we use Mertens' theorem in the following forms.
For $x \ge 2$, 
\begin{align}
 \label{eq:5.1}
  \sum_{p \le x} \frac{1}{p} 
 = \log\log x + \gamma + \Oh[(\log x)^{-1}],
\end{align}
and
\begin{align}
 \label{eq:5.2}
  \prod_{p \le x}\br{1 - \frac{1}{p}}
 = \frac{\e^{-\gamma}}{\log x}\br{1 + \Oh[\frac{1}{\log x}]},
\end{align}
where  $\gamma = 0.5772\ldots$  is the Euler--Mascheroni constant.
Second, we use a bound for the number of  $y$-smooth  numbers less 
than or equal to $x$, that is for
\[
 \Psi(x,y) = |\{ n \le x : p \mid n \implies p \le y \}|.
\]
Namely, as a  consequence of  \cite[Theorem III.5.1]{TEN}, we have
\begin{align}
 \label{eq:5.3}
  \Psi(x,y) \ll x(\log x)^{-1} 
   \qquad (1 \le 2\log y \le (\log x)(\log_2 x)^{-1}).
\end{align}
Third, we use the prime number theorem for arithmetic progressions 
in the following form due to Page  (see \cite[\S20, (13)]{DAV} and 
also the proof of \eqref{eq:4.7} above).
Let $c$ be any positive constant.
There is a positive constant $c'$, which is determined by $c$, 
such that 
\begin{align}
 \label{eq:5.4}
  \sums[x < p \le x + y][p \equiv a \bmod q] \log p
   = \frac{y}{\phi(q)} + \Oh[x\exp(-c'\sqrt{\log x})]  
\end{align}
uniformly for $2 \le y \le x$, $q \le \exp\br{c\sqrt{\log x}}$ and 
$(q,a) = 1$, except possibly if $q$ is a multiple of a certain 
integer $q_1$ depending on $x$ which, if it exists, satisfies 
$\gp{q_1} \gg \log_2 x$ \textup{(}the implicit constant also 
determined by $c$\textup{)}.

\begin{lemma}
 \label{lem:5.2}
Fix an integer $k \ge 1$ and real numbers 
$\beta_k \ge \cdots \ge \beta_1 \ge 0$.
There is a number $y(\bb,k)$, depending only on 
$\beta_1, \ldots, \beta_k$ and $k$, such that the following holds.
Let $x,y,z$ be any numbers satisfying  $x \ge 1$, $y \ge y(\bb,k)$ 
and 
\begin{align}
 \label{eq:5.5}
  2y(1 + (1 + \beta_k)x)
   \le 2z 
    \le y(\log_2 y)(\log_3 y)^{-1}.
\end{align}
Let $\Z$ be any  \textup{(}possibly empty\textup{)} set of  primes 
such that for all $p' \in \Z$, 
\begin{align}
 \label{eq:5.6}
  \sums[p \in \Z][p \ge p'] \frac{1}{p} 
   \ll \frac{1}{p'} 
    \ll \frac{1}{\log z}. 
\end{align}
There  is  a  set  $\{ a_p : p \le y, \, p \not\in \Z \}$  and  an 
admissible $k$-tuple $\{h_1,\ldots,h_k\}$ such that 
\begin{align}
 \label{eq:5.7}
  \textstyle
   \{h_1,\ldots,h_k\} 
  = \br{(0,z] \cap \ZZ}
     \setminus
      \bigcup_{p \le y, \, p \not\in \Z} 
       \{g : g \equiv a_p \bmod p\}.  
\end{align}
Moreover, for $1 \le i < j \le k$, 
\begin{align}
 \label{eq:5.8}
  p \mid (h_j - h_i) \implies p \le y,
\end{align}
and for $1 \le i \le k$, 
\begin{align}
 \label{eq:5.9}
  h_i = \beta_i x y + y
       + \Oh\big(y\e^{-(\log y)^{1/4}}\big).
\end{align}
\end{lemma}
\begin{proof}
Let $y_1$, $y_2$, $y$ and $z$ be numbers such that 
\begin{align}
 \label{eq:5.10}
  2 < y_1 < y_2 < y < z < y_1y_2          \quad \text{and} \quad 
  2\log y_1 \le (\log z)(\log_2 z)^{-1}. 
\end{align}
Let  $\Z$  be  any  set  of   primes  satisfying  \eqref{eq:5.6}.
We  assume that $2 \not\in \Z$ (which follows from \eqref{eq:5.6} 
if $y$ [and hence $z$] is large enough).
Let
\[
 P_1 = \prods[2 < p \le y_1][p \not\in \Z, \, p \ne \ell] p, \quad 
 P_2 = \prods[y_1 < p \le y_2][p \not\in \Z] p,              \quad 
 P_3 = \prods[y_2 < p \le y][p \not\in \Z] p,
\]
where in the definition of $P_1$, $\ell$ is a prime satisfying 
$\ell \gg \log y_1$.
(We will eventually specify  $\ell$  according to \eqref{eq:5.4}, 
but for the time being it can be treated as arbitrary.)
It is important to note that $2 \nmid P_1$.

We record three  bounds related to  $\Z$,  which  all  follow from 
\eqref{eq:5.6}.
First,  using  the  notation  $(n,\Z)  \ne  1$  to  indicate  that 
$p \mid n$ for some $p \in \Z$, we have
\begin{align}
 \label{eq:5.11}
  \sums[n \le z][(n,\Z) \ne 1] 1 
   \le
    \sum_{p \in \Z} \left[\frac{z}{p}\right] 
     \ll \frac{z}{\log z}.
\end{align}
Second, we have 
\[
 \sum_{p \in \Z} \log\br{\frac{p}{p - 1}}
  \le \sum_{p \in \Z} \frac{1}{p - 1}
   \ll \frac{1}{\log z},
\]
hence (upon exponentiation),
\begin{align}
 \label{eq:5.12}
  \prod_{p \in \Z} \br{1 - \frac{1}{p}}^{-1}
   = 1 + \Oh[\frac{1}{\log z}].
\end{align}
Third, since 
$\sum_{p \in \Z, \, p \ge p'} 1/p \ll 1/p'$  for all  $p' \in \Z$, 
the  elements  of  $\Z$  grow  at  least  as  fast  as a geometric 
progression, hence for all $y_0 \ge 1$, 
\begin{align}
 \label{eq:5.13}
  \sums[p \in \Z][p \le y_0] 1 
   \ll \log y_0.
\end{align}

For $p \mid P_2$ we choose $a_p = 0$.
Thus, letting 
\[
 \textstyle
 \Res_1 
  = \br{(0,z] \cap \ZZ}
   \setminus
    \bigcup_{p \,\mid\, P_2} \{g : g \equiv a_p \bmod p\} 
     = \{h \in (0,z] : (h,P_2) = 1\},
\]
it  is  clear  that  $h \in \Res_1$  only  if at  least one of the
following  holds: 
\begin{enumerate}[label=(\roman*)]
 \item $(h,\Z) \ne 1$;
 \item $h$ is $y_1$-smooth;  
 \item $h = pm$ for some prime $p > y_2$ and positive integer 
       $m \le z/p$.
\end{enumerate}
In   case  (iii),  the  prime  $p$  is uniquely  determined  since 
$z < y_1y_2 < y_2^2$.
Therefore,   by   \eqref{eq:5.11},   the   smooth   number   bound 
\eqref{eq:5.3} and Mertens' theorem \eqref{eq:5.1},
\[
 |\Res_1| 
  \le 
   \sums[h \le z][(h,\Z) \ne 1] 1
    + 
     \Psi(z,y_1) 
    + \sum_{y_2 < p \le z} \left[\frac{z}{p}\right]
  =  z\log\br{\frac{\log z}{\log y_2}} 
     + \Oh[\frac{z}{\log y_2}].
\]
Taking into account that 
\[
  \log\bigg(\frac{\log z}{\log y_2}\bigg)
 = \log\bigg(1 + \frac{\log(z/y_2)}{\log y_2}\bigg)
    \le \frac{\log(z/y_2)}{\log y_2},
\]
it follows that
\begin{align}
 \label{eq:5.14}
 |\Res_1| 
   \le \frac{z}{\log y_2}
        \br{\log(z/y_2) + \Oh[1]}.
\end{align}

For  $p  \mid  P_1$  we  choose  $a_p$  ``greedily''  as  follows.
For any finite set $S$ of integers and any prime $p$, 
\[
 |S| = \sum_{a \bmod p} \hspace{3pt} 
        \sums[g \in S][g \equiv a \bmod p] 1,
\]
so there exists an integer $a_p$ such that 
$
 |\{g \in S : g \equiv a_p \bmod p\}| \ge |S|/p.
$
We select a prime $p \mid P_1$ and choose $a_p$ so that this holds 
with $\Res_1$ in place of $S$.
Repeating this process one prime at a time, with  $p$ varying over 
the prime divisors of $P_1$, we eventually obtain a set
\[
 \textstyle
 \Res_2
  = \Res_1
     \setminus 
      \bigcup_{p \,\mid\, P_1} \{g : g \equiv a_p \bmod p\}
\] 
whose cardinality satisfies the bound
\begin{align}
 \label{eq:5.15}
   |\Res_2|
    \le |\Res_1| 
     \prod_{p \,\mid\, P_1}\br{1 - \frac{1}{p}} 
      \le 2\e^{-\gamma}
           \frac{z\br{\log(z/y_2) 
          + \Oh[1]}}{(\log y_1)(\log y_2)}.           
\end{align}
The   last   bound   follows   by   combining   Mertens'   theorem 
\eqref{eq:5.2}, \eqref{eq:5.12} and \eqref{eq:5.14}.
(Recall  that $2 \nmid P_1$, $\ell \nmid P_1$, $\ell \gg \log y_1$ 
and $\log(z/y_2) < \log y_1$.)

Now, by the prime number theorem,  
\begin{align*}
  \pi(y) - \pi(y_2) 
    & = 
   \frac{y}{\log y} 
           + \Oh[\frac{y}{(\log y)^2} + \frac{y_2}{\log y_2}] 
 \\ & 
    \ge 
     \frac{y}{\log y_2} 
    + \Oh[\frac{y_2}{\log y_2} + \frac{y}{(\log y_2)(\log y)}].
\end{align*}
Combining this with \eqref{eq:5.13} and \eqref{eq:5.15}, we obtain
\begin{align}
 \label{eq:5.16}
  \begin{split}
 |\{p \in (y_2,y] : p \not\in \Z\}| - |\Res_2|
  & \ge 
     \frac{y}{\log y_2}
       \br{1 - 2\e^{-\gamma}
        \frac{z\log(z/y_2)}{y\log y_1}} \\
  & \hspace{30pt} + 
     \Oh[\frac{y_2}{\log y_2} 
        + \frac{z}{(\log y_1)(\log y_2)}].
  \end{split}
\end{align}
We will presently require that $y_1 \le c\sqrt{\log y}$, so we now 
assume that
\[
 y_1 = (\log y)^{1/4}, \quad 
 y_2 = y(\log_3 y)^{-1}, \quad 
 y < 2z \le y(\log_2 y)(\log_3 y)^{-1}.
\]
Then by \eqref{eq:5.16} we have 
\[
 |\{p \in (y_2,y] : p \not\in \Z\}| - |\Res_2|
   \ge \frac{y}{\log y}
        \br{1 - \e^{-\gamma}}
               + \Oh[\frac{y}{(\log y)(\log_3 y)}].
\]
The right-hand side tends to infinity with $y$, and so 
\[
 |\{p \in (y_2,y] : p \not\in \Z\}| > |\Res_2| + k
\]
if $y$ is sufficiently large in terms of $k$, as we now assume.

Applying Lemma \ref{lem:5.2}  we see that if  $\{h_1,\ldots,h_k\}$ 
is an arbitrary admissible  $k$-tuple contained in  $\Res_2$, then 
there are integers $\{a_p : p \mid 2\ell P_3\}$ such that 
\[
   \{h_1,\ldots,h_k\} 
  = \Res_2
     \setminus
      \textstyle 
       \bigcup_{p\,\mid\, 2\ell P_3} \{g : g \equiv a_p \bmod p\}. 
\]
Therefore, since   
$\{p \le y : p \not\in \Z\} 
= \{p \le y : p \mid 2\ell P_1P_2P_3\}$,
to  complete  the  proof  it  suffices  to show  that  there is an 
admissible  $k$-tuple  $\{ h_1, \ldots, h_k \}  \subseteq  \Res_2$ 
satisfying \eqref{eq:5.8} and \eqref{eq:5.9}.

To this  end, let  $A \bmod P_1$  be  the  arithmetic  progression 
modulo $P_1$ such that for all $p \mid P_1$,
\begin{align*}
 A \equiv 
    \begin{cases}
     -1           & \text{if $a_p \equiv 1 \bmod p$,}     \\
     \phantom{-}1 & \text{if $a_p \not\equiv 1 \bmod p$.}
    \end{cases}
\end{align*}
(Recall that  $2 \nmid P_1$, so $-1 \not\equiv 1 \bmod p$  for all 
$p \mid P_1$.)
Then  $(A, P_1) =  1$  and   the   primes   $h  \in  (y,z]$   with 
$h \equiv A \bmod P_1$ all lie in $\Res_2$.
(If   $h  \in  (y,z]$   is   prime  then  $(h, P_2)  =  1$,  hence 
$h \in \Res_1$.)
We choose the elements of our  $k$-tuple  from among those primes.
We  note that by the prime  number  theorem  and  \eqref{eq:5.13}, 
$P_1  =  \e^{ (1  +  \oh(1))y_1 }$ as $y$  (and hence $y_1$) tends 
to infinity.
Thus,  if  $h$  and  $h' < h$  are   any  two   such  primes  then
\[
 p \mid h - h'                               \quad \implies \quad 
 \text{$p \mid P_1$ or $p \mid(h - h')/P_1$} \quad \implies \quad 
 p \le \max\{y_1,z/P_1\} < y 
\]
if $y$  is large enough, as we assume, so any  $k$-tuple of primes 
$\{h_1,\ldots,h_k\}$ chosen in this way satisfies \eqref{eq:5.8}.
Moreover,  such  a   $k$-tuple  of  primes  is  admissible   since 
$\min\{h_1,\ldots,h_k\} > k$ (we assume that $y > k$).

By Chebyshev's bound we have 
$\sum_{p \le y_1} \log p < 2y_1$, whence 
$P_1 < \e^{2(\log y)^{1/4}}$.
Thus, by \eqref{eq:5.4} we have 
\[
 \sums[u < p \le u + \Delta][p \equiv A \bmod P_1] \log p
  = \frac{\Delta}{\phi(P_1)} 
     + \Oh[y\exp\br{-c'\sqrt{\log y}}],
\]
uniformly for  $2 \le \Delta \le y \le u \le z$, where  $c'$ is an 
absolute  constant, except  possibly if  $P_1$  is a multiple of a 
certain  integer  $q_1$  whose  greatest  prime  factor  satisfies  
$\gp{q_1} \gg \log_2 y \gg \log y_1$.
We now specify $\ell$ accordingly so that  this possibility cannot 
arise.%
\footnote{%
If $q_1$ does not exist we can either let $\ell = 1$ or choose any 
$\ell \gg \log y_1$.
Indeed,  we could remove any set $\Z_1$ of primes  from $P_1$ such 
that $\sum_{\ell \in \Z'} 1/\ell \ll 1/\log y_1$, without 
affecting the proof.
}
We let
$
 \Delta = y\e^{-(\log y)^{1/4}}.
$
Thus, 
\[
 \sums[u < p \le u + \Delta][p \equiv A \bmod P_1] \log p
   \gg y\exp\br{-3(\log y)^{1/4}}  
\]
uniformly for  $y \le u \le z$, and  the  left-hand  side is a sum 
over at least $k$ primes if  $y$ is sufficiently large in terms of 
$k$, as we now assume. 

Recall that  $\beta_k \ge \cdots \ge \beta_1 \ge 0$ are given real 
numbers.
We now assume that  $y$  is large enough in terms of  $\beta_k$ so 
that 
\[
 2(1 + (1 + \beta_k)) \le (\log_2 y)(\log_3 y)^{-1},
\]
and we let $x$ be any number such that $x \ge 1$ and 
\[
 2y(1 + (1 + \beta_k)x) 
   \le 2z 
    \le y(\log_2 y)(\log_3 y)^{-1}.
\]
For $1 \le i \le k$, let 
\[
 u_i = \beta_i xy + y,
\]
so that the intervals  $(u_i, u_i + \Delta]$  are all contained in 
$(y,z]$.
For   each  $1  \le i  \le  k$   in   turn,  we   choose  a  prime 
$h_i \in (u_i, u_i + \Delta]$  with  $h_i \equiv A \bmod P_1$ a nd 
$h_i \ne h_j$ for any $j \le i$. 
This is possible since each interval contains at least  $k$ primes 
that are congruent to $A \bmod P_1$. 
We see that the resulting set  $\{h_1,\ldots,h_k\}$  is admissible 
since  no  element is  congruent  to  $a_p \bmod p$  for any prime 
$p \le k$.  
Moreover, $h_i = u_i + \Oh[\Delta]$, which  gives \eqref{eq:5.9}.
\end{proof}

%%%%%%%%%%%%%%%%%%%%%%%%%%%%%%%%%%%%%%%%%%%%%%%%%%%%%%%%%%%%%%%%%%
%%%%%%%%%%%%%%%%%%%%%%%%%%%  SECTION 6 %%%%%%%%%%%%%%%%%%%%%%%%%%%
%%%%%%%%%%%%%%%%%%%%%%%%%%%%%%%%%%%%%%%%%%%%%%%%%%%%%%%%%%%%%%%%%%

\section{Deduction of Theorems \ref{thm:1.1} and \ref{thm:1.3}}
 \label{sec:6}

\subsection*{Deduction of Theorem \ref{thm:1.3}}
 \label{subsec:6.1}

Fix $k \ge m \ge 2$ and $\epsilon = \epsilon(k,m) \in (0,1)$, 
with $k$ a sufficiently  large multiple of $(8m^2 + 8m)(8m + 1)$, 
and $\epsilon$ sufficiently small, in the sense of Theorem 
\ref{thm:4.3}.

Fix real numbers $\beta_{8m^2+8m} \ge \cdots \ge \beta_1 \ge 0$. 
Let $\bb\in\RR^k$ be given by 
\[
 \bb 
  = 
   (\beta_1,\ldots,\beta_1,
     \beta_2,\ldots,\beta_2,
      \ldots,
       \beta_{8m^2 + 8m},\ldots, \beta_{8m^2 + 8m}),
\]
where there are $k/(8m^2 + 8m)$ consecutive copies of each $\beta_i$ 
appearing in $\bb$.
Let $N \ge N(k,m,\epsilon)$ (as in Theorem \ref{thm:4.3}) and  put
\[
 x = \epsilon^{-1},               \quad 
 y = w = \epsilon\log N,          \quad 
 z = y(\log_2 y)(2\log_3 y)^{-1}.
\]
If  $N \ge N(\bb,k,m,\epsilon)$  is large enough in terms of 
$\bb$ and $k$, then with $y(\bb,k)$ as in Lemma \ref{lem:5.2} we 
have
\[
 x > 1,                                                   \quad 
 y \ge y(\bb,k),                                          \quad
 2y(1 + (1 + \beta_k)x) \le 2z \le y(\log_2 y)(\log_3 y)^{-1}.  
\]
Let   $\Z[ N^{4\epsilon} ]$   be  given  by \eqref{eq:4.8} and let  
$W = \prod_{p \le w,\, p \, \nmid \, \Z[N^{4\epsilon}]} p$. 
Let us define $\Z$ by putting $\Z = \emptyset$ if 
$\Z[ N^{4\epsilon} ] = 1$ and $\Z = \Z[ N^{4\epsilon} ]$ if 
$\Z[ N^{4\epsilon} ] \ne 1$.
Then \eqref{eq:4.7} implies that the condition \eqref{eq:5.6} is 
satisfied since $\log z \ll \log_2 N^{\epsilon}$.

The hypotheses of Lemma \ref{lem:5.2}  being verified, we conclude 
that there exists a set  $\{a_p : p \le y, \, p \not\in \Z\}$  and 
an admissible $k$-tuple $\{h_1,\ldots,h_k\}$ such that 
\begin{align}
 \label{eq:6.1}
 \textstyle
 \{h_1,\ldots,h_k\}
  = \br{(0,z] \cap \ZZ}
     \setminus
      \bigcup_{p \le y,\, p \not\in \Z}
       \{g : g \equiv a_p \bmod p\}. 
\end{align}
Moreover, for $1 \le i < j \le k$, 
\begin{align}
 \label{eq:6.2}
  p \mid (h_j - h_i) \implies p \le y = w
\end{align}
and we have the partition
\begin{align}
 \label{eq:6.3}
  \A = \A_1 \cup \cdots \cup \A_{8m^2 + 8m}
\end{align}
such that for each $j \in \{1,\ldots,8m^2 + 8m\}$ we have
\begin{align}
 \label{eq:6.4}
 h = (\beta_j + \epsilon + \oh[1])\log N
   \quad 
    \text{for all $h \in \A_j$.}
\end{align}
We let $b$ be an integer satisfying
\begin{align}
 \label{eq:6.5}
 b \equiv -a_p \bmod p
\end{align}
for all $p \le y$, $p \not\in \Z$.

We now wish to apply part (ii) of Theorem \ref{thm:4.3}. 
We have  $0 < h_i \le z < N$  for each  $i$, so \eqref{eq:4.16} is
satisfied.
We see \eqref{eq:6.2} and \eqref{eq:6.3} give the conditions 
\eqref{eq:4.17} and \eqref{eq:4.18}. 
Finally,   by   \eqref{eq:6.1}   and   \eqref{eq:6.5},   we   have 
$(\prod_{i = 1}^k(b + h_i), W)  =  1$, and so \eqref{eq:4.19} also 
holds. 
We  conclude   that   there   exists   some   $n \in (N,2N]$  with 
$n \equiv b \bmod W$ and some set $\{i_1,\dots,i_{m+1}\}$ such that
\begin{align*}
 | \A_i(n) \cap \PP | 
   & = 1
   \quad
    \text{for all $i \in \{i_1,\ldots,i_{m + 1}\}$},
 \\
 | \A_i(n) \cap \PP | 
   & \le 1
     \quad
      \text{for all $i_1<i<i_{m+1}$.}
\end{align*}

For any $n  >  y$ such that $n  \equiv  b \bmod W$, \eqref{eq:6.1} 
implies that 
\[
 (n,n + z] \cap \PP = \A(n) \cap \PP,
\]
because if  $g \in (0,z]$  and  $g \not\in \{h_1,\ldots,h_k\}$, we 
have   $g + n  \equiv  a_p  -  a_p  \equiv 0  \bmod p$  for   some  
$p \le w$ with $p \in \Z$. 
The primes in $\A(n)$  are therefore  consecutive 
primes.
Therefore there are indices  $J(1)  <  \cdots  <  J(m+1)$  for 
which $|\A_{J(i)}(n) \cap \PP| = 1$ and the primes counted here 
form a sequence of $m + 1$ consecutive primes.
Thus, by  \eqref{eq:6.4},  and  since  $N \le n + h_i \le 3N$, we 
have for some $r$ that 
\begin{align}
 \label{eq:6.6}
  \frac{ p_{r + i + 1} - p_{r+i} }{\log{p_{r+i}} } 
 = \beta_{J(i+1)} - \beta_{J(i)} + \oh[1],
\end{align}
for $1 \le i \le m$.

Letting  $N$ tend to infinity, we see that for infinitely many $r$  
there exists some  $1 \le J(1) <  \cdots  < J(m+1) \le 8m^2+8m$ 
such that \eqref{eq:6.6} holds.
Since there are at most $\Oh[1][k]$ distinct ways to choose the
indices $J(i)$, at least one pattern of indices occurs infinitely 
often.
For that pattern we have  \eqref{eq:6.6}  for infinitely many $r$, 
and so 
$
 (\beta_{J(2)} - \beta_{J(1)},
   \ldots,
    \beta_{J(m+1)} - \beta_{J(m)})
     \in 
      \LP_m.
$
\qed

\subsection*{Deduction of Theorem \ref{thm:1.1}}
 \label{subsec:6.2}
The argument is essentially the same as that for Theorem 
\ref{thm:1.3}, but uses part (i) Theorem \ref{thm:4.3} instead of 
part (ii). 

We take $k$ to be a sufficiently large multiple of $9\times 17$. 
Given $\beta_9 \ge \cdots \ge \beta_1 \ge 0$, we construct $\A$ as 
before and form a partition $\A = \A_1 \cup \cdots \cup \A_9$, so 
that each $\A_i$ has size $k/9$ and all elements of $\A_i$ have 
size $(\beta_i + \epsilon + \oh[1])\log N$.
Applying part (i) of Theorem \ref{thm:4.3} (with $m = 1$) then 
shows that there is an $n \in (N,2N]$ such that 
$|\A_i(n) \cap \PP| \ge 1$, $|\A_i(n) \cap \PP| \ge 1$ for some 
$1 \le i < j \le 9$. 
As before, our construction shows that there are no other primes 
in $[n,n+z]$, and so there must be two consecutive primes $p_r$, 
$p_{r+1}$ of the form $n+h,n+h'$ with $h,h'$ in different sets 
$\A_i$. 
But then we have
\[
  \frac{ p_{r + 1} - p_{r} }{\log{p_{r}} } 
  = \beta_j - \beta_i + \oh[1],
\]
for some $i < j$. 
Since this occurs for every large $N$, we obtain the result.
\qed

%%%%%%%%%%%%%%%%%%%%%%%%%%%%%%%%%%%%%%%%%%%%%%%%%%%%%%%%%%%%%%%%%%
%%%%%%%%%%%%%%%%%%%%%%%%%%%  SECTION 7 %%%%%%%%%%%%%%%%%%%%%%%%%%%
%%%%%%%%%%%%%%%%%%%%%%%%%%%%%%%%%%%%%%%%%%%%%%%%%%%%%%%%%%%%%%%%%%

\section{Concluding remarks}
 \label{sec:7}

If the statement of  Theorem \ref{thm:4.2}  held with an arbitrary 
fixed $\theta \in (0,1)$, then one could apply a minor 
adaptation of the Maynard--Tao argument to show that given 
$\beta_1,\ldots,\beta_5$, there are infinitely many $n$ such that 
at least two of the integers in 
$
 \{n + h_1,\ldots,n + h_5\}
$ 
are primes with $h_i \approx \beta_i\log{n}$, and so we could take 
$k = 5$ in Theorem \ref{thm:1.1}.
This would give  $\lambda([0,T] \cap \LP) \ge (1 - \oh[1]) \, T/4$
as $T \to \infty$, and $\lambda([0,T] \cap \LP) \ge 3T/25$ for all 
$T \ge 0$, in place of \eqref{eq:1.3} and \eqref{eq:1.4}.
 
We can replace the logarithm in  \eqref{eq:6.6}, hence in Theorems 
\ref{thm:1.1} and \ref{thm:1.3}, by a function 
$f  :  [N_0,\infty)  \to  [1,\infty)$ that is a monotone, strictly 
increasing,   unbounded  and   satisfies  $f(N)  \le  \log N$  and 
$f(2N) - f(N) \ll 1$ for all $N \ge N_0$.
In fact we can let  $f(N)/\log N$ tend to infinity slowly (as fast 
as $\log_3 N/\log_4 N$). 
It is  possible to improve  upon this, and it would be of interest 
to see how  fast  $f(N)$  can  grow  while  Theorem  \ref{thm:1.1} 
remains valid.
This question has recently been addressed by Pintz \cite{PIN2}.

%%%%%%%%%%%%%%%%%%%%%%%%%%%%%%%%%%%%%%%%%%%%%%%%%%%%%%%%%%%%%%%%%%
%%%%%%%%%%%%%%%%%%%%%%%%%%% REFERENCES %%%%%%%%%%%%%%%%%%%%%%%%%%%
%%%%%%%%%%%%%%%%%%%%%%%%%%%%%%%%%%%%%%%%%%%%%%%%%%%%%%%%%%%%%%%%%%

\end{document}